\documentclass[10pt,a4paper]{article}
\usepackage{amsopn}

\usepackage{graphics}
\usepackage{graphicx}
\usepackage{epsfig}

\usepackage{amssymb}
\usepackage{amsthm}
\usepackage{amsmath}
\usepackage{amsfonts}
\usepackage{lipsum}
\usepackage[left=2.5cm,right=2.5cm,bottom=3.5cm,top=3.5cm]{geometry}
\usepackage{empheq}
\usepackage{epstopdf}
\usepackage{changepage}
\usepackage{rotating}
\usepackage{adjustbox}
\usepackage{graphbox}
\usepackage{color}
\usepackage{xcolor}
\usepackage{dsfont}
\usepackage{dutchcal}
\usepackage{hyperref} % referencias cruzadas 
\usepackage{float}
\usepackage{resizegather}
\usepackage{multicol}
\usepackage{booktabs} % mejora la calidad de las tablas y proporciona nuevas opciones
\usepackage{colortbl} % coloreado en tablas
\usepackage{multirow} % tablas con celdas de varias lineas
\usepackage{longtable,tabu} % tablas
\usepackage{hhline}   % mejora las lineas horizontales en las tablas
\usepackage{anyfontsize} % tamanos de fuente personalizados
\usepackage{everysel}
\usepackage{psfrag}
\usepackage{titlesec}
\usepackage{bbm}
\usepackage{bm}
\usepackage{pgf,tikz}

\newtheorem{teorema}{Theorem}

\newtheorem{remark}[teorema]{Remark}

\allowdisplaybreaks 
\renewcommand{\arraystretch}{2}
\hypersetup{
	pdftitle={Design and analysis of ADER-type schemes for model advection-diffusion-reaction equations},    % title
	pdfauthor={S. Busto, E.F. Toro, M.E. Vazquez-Cendon},     % author
	pdfcreator={S. Busto, E.F. Toro, M.E. Vazquez-Cendon},   % creator of the document
	colorlinks=true,       % false: boxed links; true: colored links
}

\title{Design and analysis of ADER-type schemes for model advection-diffusion-reaction equations}
\begin{document}

\begin{center}
	\textbf{ \Large{Design and analysis of ADER-type schemes for model advection-diffusion-reaction equations} }
	
	\vspace{0.5cm}
	{S. Busto\footnote{saray.busto@usc.es}, E. F. Toro\footnote{eleuterio.toro@unitn.it}, M. E. V\'azquez-Cend\'on\footnote{elena.vazquez.cendon@usc.es}}
	
	\vspace{0.2cm}
	{\small
		\textit{$^{(1,3)}$ Departamento de Matem\'atica Aplicada, Universidade de Santiago de Compostela. Facultad de Matem\'aticas, ES-15782 Santiago de Compostela, Spain}
		
		\textit{$^{(2)}$ Department of Civil, Environmental and Mechanical Engineering, University of Trento, Via Mesiano 77, 38123 Trento, Italy}
	}
\end{center}

% % % % % % % % % % % % % % % % % % % % % % % % % % % % % %
%                   Abstract                              %
% % % % % % % % % % % % % % % % % % % % % % % % % % % % % %
\hrule
\vspace{0.4cm}

\begin{center}
	\textbf{Abstract}
\end{center} 

\vspace{0.1cm}
We construct, analyse and assess various schemes of second order of accuracy in space and time for model advection-diffusion-reaction differential equations.  The constructed schemes are meant to be of practical use in solving 
industrial problems and are derived following two related approaches, namely ADER and MUSCL-Hancock. Detailed analysis of linear stability and local truncation error are carried out. In addition, the schemes are implemented and assessed for various test problems. Empirical convergence rate studies confirm the theoretically expected accuracy in both space and time.

% % % % % % % % % % % % % % % % % % % % % % % % % % % % % %
%                   Keywords                              %
% % % % % % % % % % % % % % % % % % % % % % % % % % % % % %
\vspace{0.2cm}

\vspace{0.4cm}

\hrule

% % % % % % % % % % % % % % % % % % % % % % % % % % % % % %
%                 Introduction                            %
% % % % % % % % % % % % % % % % % % % % % % % % % % % % % %
\section{Introduction}
Advection-diffusion-reaction equations are present in a wide range of physical/biological problems. The Navier-Stokes equations are a prominent example, which constitute a major focus for the development of numerical methods of practical use to the scientific community. A classical and successful approach for solving the aforementioned and related equations are finite volume methods (e.g. see  \cite{Toro}, \cite{BFSV14}, \cite{GR96}, \cite{LV02} and references herein).

A motivation of this paper concerns the simulation of low Mach number flows in a turbulent regime. The phenomenon can be represented by the Navier-Stokes equations coupled to a RANS $k-\varepsilon$ turbulent model, see \cite{BBCFSV15}. 
The approach introduces turbulent viscosity, which is typically computed by solving an additional pair of advection-diffusion-reaction equations, that is equations for the turbulent kinetic energy and dissipation rate. One issue here is the time dependency of the viscous term. This requires the use of methods that are at least second-order accurate in space and time for all terms involved. In practice, it is often the case that the numerical methods used are of low order of accuracy. Typically, methods may be of second order in space but only first order in time, or may be second order in both space and time but only for some of the terms in the equations.
For diffusion equations a popular choice is the second-order Crank-Nicolson method, {\color{black} see \cite{SO88}}. The accuracy for reaction terms, coupled to the remaining terms of the equations, is usually sacrificed, resulting in overall low order of accuracy.

For advection equations, several approaches for constructing high-order methods have been put forward. A classical example is the Lax-Wendroff scheme \cite{LW60}, \cite{Lax57}. This scheme is linear in the sense of Godunov \cite{God59} and thus oscillatory, according to Godunov's theorem \cite{God59}. We note that the oscillatory nature remains so even when (physical) viscous terms are added. A major step forward in this direction was the work of Kolgan \cite{Kol11}, who introduced, for the first time, a numerical scheme that circumvents Godunov's theorem, via the construction of a non-linear scheme using non-linear reconstructions (limited slopes), see \cite{Ber06} and \cite{CV12} for further details. Since then, many more works have appeared in the literature, reporting schemes, such as Total Variation Diminishing Methods (TVD) and {\color{black}Flux Limiter Methods} (see, for instance, \cite{Leer84}, \cite{VL97} and \cite{Swe84}). Comprehensive reviews are found in  \cite{Toro} and \cite{LV02}, for example.

More advanced non-linear methods for advection dominated problems include the semi-discrete ENO, {\color{black} see \cite{HEOC87}}, and WENO, {\color{black} see \cite{SO88} and \cite{LSC94},} approaches. See also the method of Harten and collaborators  \cite{HEOC87}, which is a fully discrete high-order scheme. In \cite{CT08}, this scheme was called the HEOC scheme, and was re-interpreted in terms of the solution of a generalised Riemann problem, solved in a particular way. If fact, it is easily shown that the HEOC scheme is a generalisation of the MUSCL-Hancock method, {\color{black} see \cite{Leer84}}.

The ADER approach, first put forward in \cite{TMN01}, is {\color{black}  also a} fully discrete approach that relies on non-linear reconstructions and the solution of the generalised Riemann problem, to any order of accuracy. The resulting schemes are arbitrarily accurate in both space and time in the sense that they have no theoretical accuracy barrier. An introduction to ADER schemes is found in Chapters 19 and 20 of \cite{Toro}.  Further developments and applications are found, for example, in   \cite{TT05}, \cite{TDTK06}, \cite{TT06}, \cite{TT06ecc}, \cite{Tit05}, \cite{TT07},
\cite{Col90}, \cite{Tak06}, \cite{TT05tvd}, \cite{CSDT09}, \cite{Zah08}, \cite{BD14}, \cite{IG14}, \cite{TT04}, 
\cite{MT14},  \cite{MG11}, \cite{Her02}, \cite{CM14}, \cite{CMNP13}, \cite{OA02}, \cite{DM05},  \cite{Gas07}, 
\cite{Dum10},  \cite{HD11}.

The aim of this paper is to develop  finite volume schemes of second-order in time and space to solve the advection-diffusion-reaction equation, admitting space and time dependent diffusion {\color{black} coefficients}. We follow the ADER and MUSCL-Hancock methodologies and compare both approaches. Detailed analysis, such as linear stability and accuracy in the sense of local truncation error is lacking for these methodologies applied to advection-diffusion-reaction equations. The main objective of this paper is precisely to carry out detailed stability and accuracy analysis of these methods. {\color{black}
Moreover, to determine the stability region and a new graphical methodology is introduced.}

The outline of this paper is as follows. In Section \ref{sec:adreeq} the advection-diffusion-reaction equation is introduced.  Section \ref{sec:ader} is devoted to the development of a numerical scheme for the advection-diffusion-reaction equation. The  ADER approach is adopted to approximate linear advection-reaction equations and is modified  to account for the diffusion term.
The methodology developed for the ADER scheme to treat source-term like terms is applied to the MUSCL-Hancock method in Section \ref{sec:muscl_hancock}.  In Section \ref{sec:stability}, we conduct a linear stability analysis of the schemes.
Section \ref{sec:numerical_results} is devoted to the study of empirical convergence rates of the scheme.  
Conclusions are drawn in Section \ref{sec:conclusions} and \ref{sec:appendix} is devoted to the analysis {\color{black} of} local truncation errors of the proposed schemes.

% % % % % % % % % % % % % % % % % % % % % % % % % % % % % % 
% % % % % % % % % % % % % % % % % % % % % % % % % % % % % %
%                     ADRE                                %
% % % % % % % % % % % % % % % % % % % % % % % % % % % % % %
\section{The advection-diffusion-reaction equation}\label{sec:adreeq}
The advection-diffusion-reaction equation reads 
\begin{equation}
	\partial_t q (x,t) + \lambda\partial_x q(x,t) = {\color{black}\partial_x \left( \alpha \partial_x q\right)(x,t)} +\beta q(x,t)\label{eq:adre}
\end{equation}
where $q(x,t)$ is the conservative variable;
$x,t$ are the spatial and temporal independent variables; $\lambda$ is the characteristic speed;
$\alpha(x,t)$ is the diffusion coefficient, a prescribed function;
and $\beta$ is the coefficient of the reaction term (source term). 

In order to solve equation \eqref{eq:adre} we work in the finite volume framework.  To start with, we consider the control volume $V=[x_{i-\frac{1}{2}},x_{i+\frac{1}{2}}]\times[t^{n},t^{n+1}]$ in the $x-t$ plane, of dimensions,
$\Delta x= x_{i+\frac{1}{2}}-x_{i-\frac{1}{2}}$, $\Delta t= t^{n+1}-t^{n}$. Then, exact integration of equation \eqref{eq:adre} in the control volume $V$ gives  
\begin{gather*}
	\frac{1}{\Delta x}\int_{x_{i-\frac{1}{2}}}^{x_{i+\frac{1}{2}}}\!\!\left(  q(x,t^{n+1}) -q(x,t^{n})  \right) dx + \frac{1}{\Delta x}\int_{t^{n}}^{t^{n+1}}\!\! \left(  f(q(x_{i+\frac{1}{2}},t)) -   f(q(x_{i-\frac{1}{2}},t))\right)  dt\\ 
	= 
	\frac{1}{\Delta x}\int_{x_{i-\frac{1}{2}}}^{x_{i+\frac{1}{2}}}\left(  \int_{t^{n}}^{t^{n+1}} \partial_x \left(  \alpha\partial_x q\right){\color{black}(x,t)}\;   dx\right) dt  + \frac{1}{\Delta x}\int_{x_{i-\frac{1}{2}}}^{x_{i+\frac{1}{2}}}\left(  \int_{t^{n}}^{t^{n+1}} \beta q{\color{black}(x,t)} \;  dx\right) dt \;.
\end{gather*}	
Introducing the notation 
\begin{equation*}
	q_{i}^{n+1} = \frac{1}{\Delta x} \int_{x_{i-\frac{1}{2}}}^{x_{i+\frac{1}{2}}}  q(x,t^{n+1})  \;dx, \quad
	q_{i}^{n} = \frac{1}{\Delta x} \int_{x_{i-\frac{1}{2}}}^{x_{i+\frac{1}{2}}} q(x,t^{n}) \;dx,
\end{equation*}
\begin{equation*}
	f_{i+\frac{1}{2}}^{n} = \frac{1}{\Delta t} \int_{t^{n}}^{t^{n+1}}  f(q(x_{i+\frac{1}{2}},t)) \; dt,
	\quad
	f_{i-\frac{1}{2}}^{n} = \frac{1}{\Delta t} \int_{t^{n}}^{t^{n+1}}  f(q(x_{i-\frac{1}{2}},t)) \; dt,
\end{equation*}
\begin{equation*}
	g_{i}^{n} = \frac{1}{\Delta t \Delta x} \int_{x_{i-\frac{1}{2}}}^{x_{i+\frac{1}{2}}}\left(  \int_{t^{n}}^{t^{n+1}} \partial_x \left(  \alpha\partial_x q\right){\color{black}(x,t)}  \;  dx\right) dt,
\end{equation*}
\begin{equation*}
	s_{i}^{n} = \frac{1}{\Delta t \Delta x} \int_{x_{i-\frac{1}{2}}}^{x_{i+\frac{1}{2}}}\left(  \int_{t^{n}}^{t^{n+1}} \beta q{\color{black}(x,t)} \;  dx\right) dt,
\end{equation*}
we arrive at the exact relation
\begin{equation}
	q_{i}^{n+1}=q_{i}^{n}-\frac{\Delta t}{\Delta x} \left(f_{i+\frac{1}{2}}^{n}- f_{i-\frac{1}{2}}^{n} \right) + \Delta t g_i^{n}+ \Delta t s_i^{n}.
	\label{eq:fv}
\end{equation}
Thus, we can construct a numerical method to find the solution of \eqref{eq:adre} at time $t^{n+1}$ by interpreting \eqref{eq:fv} in an approximate manner, that is by approximating the integrals in \eqref{eq:fv}.
We analyse two different approaches for computing them, the ADER and the MUSCL-Hancock methodologies.
For the sake of simplicity, we will use the same notation for the approximate values than for the exact values of the integrals in \eqref{eq:fv}. It is appropriate to remark that the integral $g_{i}^{n}$ could be further integrated to yield corresponding expressions for viscous numerical fluxes.

% % % % % % % % % % % % % % % % % % % % % % % % % % % % % %
% % % % % % % % % % % % % % % % % % % % % % % % % % % % % %
%                     ADER                                %
% % % % % % % % % % % % % % % % % % % % % % % % % % % % % %
\section{The ADER approach}\label{sec:ader}
The Arbitrary high order DErivative Riemann problem (ADER) approach was first put forward in \cite{TMN01} for the linear advection equation in one and three space dimensions. In this section, we introduce a modification of the ADER approach to solve the advection-diffusion-reaction equation. The proposed method includes the following steps:
\begin{description}
	\item[Step 1.] Polynomial reconstruction.	
	Following \cite{TMN01}, we consider a reconstruction of the data in terms of first-degree polynomials of the form
	\begin{equation*}
		p_i(x)=q_i^n+\Delta_i (x-x_i),
	\end{equation*}
	where $\Delta_i$ denotes the spatial derivative of $q(x,t)$ at time $t^{n}$ in volume $i$ (or an approximation) 
	for $i=1,2,\ldots, M$, where $M$ is the total number of finite volumes.
	In the present paper, we only consider centred slopes, that is 
	\begin{equation}\Delta_i=\frac{q_{i+1}^{n}-q_{i-1}^{n}}{2\Delta x}, \label{eq:centred_slopes}\end{equation}
	which, as proved in \ref{sec:appendix}, will provide a scheme of second order accuracy in space and time. We note however that the resulting schemes will be linear and hence oscillatory in the presence of large spatial gradients.
	
	\item[Step 2.] \label{grp} Solution of the generalized Riemann problem (GRP).	
	To construct the numerical flux the following generalizations of the Classical Riemann Problem are made. On the one hand, the initial condition is assumed to be a piecewise first-degree polynomial. On the other hand, the partial differential equation accounts for the diffusion and reaction terms. That leads to the problem
	\begin{equation}
		\left\lbrace \begin{array}{l}
			\partial_t q\left(x,t\right)+\lambda \partial_x q\left(x,t\right) ={\color{black} \partial_x\left( \alpha\partial_x q\right)\left(x,t\right)}+\beta q\left(x,t\right),\\
			q(x,0)=\left\lbrace \begin{array}{ll}
				p_i(x), & x<0,\\
				p_{i+1}(x), & x>0.
			\end{array} \right.
		\end{array}\right.\label{eq:GRP}
	\end{equation}
	
	\item[Step 3.] \label{adre_difusion} \label{ader_source} Source term and diffusion term.	
	These terms are computed by approximating the integrals by the mid-point rule in both space and time.	
\end{description} 
In the following sections we will develop the last {\color{black} two} steps. 

% % % % % % % % % % % % % % % % % % % % % % % % % % % % % %
%                    STEP 2                               %
% % % % % % % % % % % % % % % % % % % % % % % % % % % % % %
\subsection{Step 2.  Solution of the generalized Riemann problem}
For ease of presentation, two different cases for the diffusion term will be considered: zero diffusion term and a space and time dependent diffusion coefficient.

% % % % % % % % % % % % % % % % % % % % % % % % % % % % % %
\subsubsection{Numerical flux without diffusion}

Expressing the solution of the GRP at the interface as a Taylor series expansion in time we have
\begin{equation} {\color{black} \overline{q}_{i+\frac{1}{2}}}=q(0,0_+) + \tau \partial_t q(0,0_+)  \;. \label{eq:taylor_lare}\end{equation}
We find the solution of \eqref{eq:GRP} as given by two terms. One being the solution of a classical Riemann problem, $q(0,0_+)$, and the other as given by the high order term, $\tau \partial_t q(0,0_+)$.

The solution of the classical Riemann problem
\begin{equation*}
	\left\lbrace \begin{array}{l}
		\partial_t q\left(x,t\right)+\lambda \partial_x q\left(x,t\right) =0,\\
		q(x,0)=\left\lbrace \begin{array}{ll}
			q_{i}^{n}+\frac{1}{2}\Delta x \Delta_i, & x<0,\\
			q_{i+1}^{n}-\frac{1}{2}\Delta x \Delta_{i+1}, & x>0,
		\end{array} \right.
	\end{array}\right.
\end{equation*} 
is
\begin{equation*}
	{\color{black}q\left(x,t\right)}= d_{i+\frac{1}{2}}^{(0)}(x/t)=\left\lbrace \begin{array}{ll}
		q_{i}^{n}+\dfrac{1}{2}\Delta x \Delta_i, & \dfrac{x}{t}<\lambda, \\
		q_{i+1}^{n}-\dfrac{1}{2}\Delta x \Delta_{i+1}, & \dfrac{x}{t}>\lambda,
	\end{array}  \right.
\end{equation*}
hence, the solution for $\frac{x}{t}=0$ is  given by
\begin{equation}
	q(0,0_+)=d_{i+\frac{1}{2}}^{(0)}(0)=\left\lbrace \begin{array}{ll}
		q_{i}^{n}+\dfrac{1}{2}\Delta x \Delta_i, & \lambda>0,\\
		q_{i+1}^{n}-\dfrac{1}{2}\Delta x \Delta_{i+1}, & \lambda<0.
	\end{array}  \right. \label{eq:ader_b1}
\end{equation}

On the other hand, regarding Cauchy-Kovalevskaya procedure and assuming a zero diffusion term, it is verified
\[\partial_t q\left(x,t\right)=-\lambda \partial_x q\left(x,t\right) +\beta q\left(x,t\right), \]
thus, we can express \eqref{eq:taylor_lare} in terms of spatial derivatives
\begin{equation}{\color{black}\overline{q}_{i+\frac{1}{2}}} = q(0,0_+) + \tau\left[-\lambda \partial_x q(0,0_+)+ \beta q(0,0_+)\right].\end{equation}
It is easy to show that the following evolution equation for the spatial derivative of the conservative variable is valid
\begin{equation}\partial_t (\partial_x q(x,t)) + \lambda \partial_x^{(2)}  q(x,t) = \beta \partial_x q(x,t) \;.
\end{equation}
Neglecting the source term, a new classical Riemann problem can be set for the spatial gradient
\begin{equation}
	\left\lbrace \begin{array}{l}
		\partial_t (\partial_x q(x,t)) + \lambda \partial_x^{(2)} q(x,t) = 0,\\
		\partial_x q(x,0) = \left\lbrace \begin{array}{ll}
			\Delta_i, & x<{\color{black} 0},\\
			\Delta_{i+1}, & x>{\color{black} 0}.
		\end{array} \right.
	\end{array} \right.
\end{equation}
Its solution is
\begin{equation}
	{\color{black}\partial_x q\left(x,t\right)}= d_{i+\frac{1}{2}}^{(1)}\left( \frac{x}{t}\right) = \left\lbrace \begin{array}{ll}
		\Delta_i, & \dfrac{x}{t}<\lambda,\\
		\Delta_{i+1}, & \dfrac{x}{t}>\lambda,
	\end{array} \right.
\end{equation}
and
\begin{equation}
	\partial_x q(0,0_{+})=d_{i+\frac{1}{2}}^{(1)}\left( 0\right) = \left\lbrace \begin{array}{ll}
		\Delta_i, & \lambda>0,\\
		\Delta_{i+1}, & \lambda<0,
	\end{array} \right. \label{eq:ader_b2}
\end{equation}
is the solution at $\frac{x}{t}=0$.

Then, from \eqref{eq:ader_b1}  and \eqref{eq:ader_b2}, the sought complete solution reads
\begin{align}
	{\color{black} \overline{q}_{i+\frac{1}{2}}} &= q(0,0_+) + \tau \partial_t q(0,0_+)\\ &=
	\left\lbrace \begin{array}{ll}
		q_{i}^{n} +\frac{1}{2} \Delta x\Delta_i + \tau \left[-\lambda \Delta_{i} +\beta \left( q_{i}^{n} +\frac{1}{2} \Delta x\Delta_i \right) \right] & \lambda>0,\notag \\
		q_{i+1}^{n} -\frac{1}{2} \Delta x\Delta_{i+1} + \tau \left[-\lambda \Delta_{i+1} +\beta \left( q_{i+1}^{n} -\frac{1}{2} \Delta x\Delta_{i+1} \right) \right] & \lambda<0.
	\end{array}  \right.
\end{align}
Performing exact integration, or equivalently applying the mid-point rule with $\tau=\frac{1}{2}\Delta t$, we obtain the numerical flux,
\begin{align}
	f_{i+\frac{1}{2}}= 
	\left\lbrace \begin{array}{ll}
		\lambda \left\lbrace q_{i}^{n} +\frac{1}{2} \Delta x\Delta_i + {\color{black} \dfrac{\Delta t}{2}} \left[-\lambda \Delta_{i} +\beta \left( q_{i}^{n} +\frac{1}{2} \Delta x\Delta_i \right) \right] \right\rbrace & \lambda>0,\notag\\
		\lambda \left\lbrace q_{i+1}^{n} -\frac{1}{2} \Delta x\Delta_{i+1} + {\color{black} \dfrac{\Delta t}{2}} \left[-\lambda \Delta_{i+1} +\beta \left( q_{i+1}^{n} -\frac{1}{2} \Delta x\Delta_{i+1} \right) \right] \right\rbrace & \lambda<0.
	\end{array}  \right.
\end{align}

% % % % % % % % % % % % % % % % % % % % % % % % % % % % % %
\subsubsection{Numerical flux with diffusion}
In order to treat the diffusion term we adopt the following strategy. The diffusion term is regarded as a source term but is introduced in the Cauchy-Kovalevskaya procedure and is evaluated in an upwind fashion. The {\color{black}upwinding} of the diffusion term can be justified because in the case of constant diffusion coefficient $\alpha$, the second derivative of the solution $q(x,y)$ of the linear advection equation, in fact any order derivative, obeys identically the same linear advection equation. Hence we can pose and solve a classical Riemann problem for these spatial derivatives leading effectively to {\it upwinding the diffusion term}. 

The Cauchy-Kovalevskaya procedure used in the Taylor series expansion in time for the solution of the GRP,
\[{\color{black}\overline{q}_{i+\frac{1}{2}}} = q(0,0_+)+\tau \partial_t q(0,0_+),\]
gives
\[{\color{black} \overline{q}_{i+\frac{1}{2}}} = q(0,0_+) + \tau\left[-\lambda \partial_x q(0,0_+) + {\color{black}\partial_x \left( \alpha\partial_x q\right)(0,0_+) } + \beta q(0,0_+)\right].\]

As anticipated previously, the term ${\color{black}\partial_x \left( \alpha\partial_x q_i\right)(0,0_+)}$ is approximated in a central difference fashion as follows
\begin{equation}
	{\color{black}\partial_x \left( \alpha\partial_x q\right)(0,0_+)} =\frac{1}{\Delta x^2}\left[ \alpha_{i+\frac{1}{2}}\left(q_{i+1}^{n}-q_{i}^{n}\right)-\alpha_{i-\frac{1}{2}}\left(q_{i}^{n}-q_{i-1}^{n}\right)  \right].\label{eq:centred_2slopes}
\end{equation}
The choice made for approximating the diffusion term in this manner is motivated by the fact that in {\color{black} our existing} 3D Navier-Stokes code in development this term will be computed by solving a pair of advection-diffusion-reaction equations.

For notational convenience we set $\left(\Delta \alpha \Delta\right)_i={\color{black}\partial_x \left( \alpha\partial_x q\right)(0,0_+)}$. 
Integrating, the numerical flux results
\begin{align}
	f_{i+\frac{1}{2}}= 
	\left\lbrace \begin{array}{ll}
		\lambda \left\lbrace q_{i}^{n} +\frac{1}{2} \Delta x\Delta_i + {\color{black} \dfrac{\Delta t}{2}} \left[-\lambda \Delta_{i} +\beta \left( q_{i}^{n} +\frac{1}{2} \Delta x\Delta_i \right) +\left( \Delta\alpha\Delta\right)_{i}\right] \right\rbrace & \lambda>0,\\
		\lambda \left\lbrace q_{i+1}^{n} -\frac{1}{2} \Delta x\Delta_{i+1} + {\color{black} \dfrac{\Delta t}{2}} \left[-\lambda \Delta_{i+1} +\beta \left( q_{i+1}^{n} -\frac{1}{2} \Delta x\Delta_{i+1} \right) +\left( \Delta\alpha\Delta\right)_{i+1}\right] \right\rbrace & \lambda<0.
	\end{array}  \right. \label{eq:flux_adre}
\end{align}

\begin{remark}
	{\color{black}From} now on, we will focus on the development and analysis of the schemes for $\lambda >0$. The results for $\lambda <0$ can be {\color{black}obtained similarly}.
\end{remark}

{\color{black} \begin{remark}
		The whole scheme can be derived from the finite volume framework, resulting in an intercell numerical flux with two terms, one for the advection and one for the diffusion, see \cite{TT01} and \cite{TH09}. However, it is worth mentioning that through an appropriated choice of the approximation for the slopes it will result in the same numerical scheme than the one proposed in this paper.
		
		Furthermore, our motivation is to couple the numerical method developed in this paper to an existing three-dimensional Navier-Stokes, see \cite{BBCFSV15}, to this end the approach proposed here turns out to be very convenient and achieves the order of accuracy sought.
	\end{remark}}
	
	% % % % % % % % % % % % % % % % % % % % % % % % % % % % % %
	%                    STEP 3                               %
	% % % % % % % % % % % % % % % % % % % % % % % % % % % % % %
	\subsection{Step 3. Approximation of the diffusion and reaction terms}
	Recall that second-order approximations to the integrals defining the finite volume method can be obtained via the mid-point rule approximation, see \cite{Toro}. This requires an approximation at the centre of the volume at the half time, which is achieved by a Taylor expansion and the use of the Cauchy-Kovalevskaya procedure, namely
	\[\overline{q_i^n}=q_i^n+\tau \left(-\lambda \Delta_i+\left(\Delta \alpha \Delta\right)_i +\beta q^{n}_i \right). \]
	
	\subsection{Step 3.1. Diffusion term}
	The approximation of the diffusion term becomes
	\begin{equation*}
		\overline{\left(\Delta \alpha \Delta\right)_i}=
		\frac{\overline{\alpha^n_{i+\frac{1}{2}}} \, \overline{\Delta_{i+\frac{1}{2}}}- \overline{\alpha^n_{i-\frac{1}{2}}}\, \overline{\Delta_{i-\frac{1}{2}}}}{\Delta x}=
		\frac{1
		}{\Delta x^2} \left[ \overline{\alpha^n_{i+\frac{1}{2}}} \left( \overline{q^n_{i+1}}-\overline{q^n_{i}}\right) 
		-\overline{\alpha^n_{i-\frac{1}{2}}} \left( \overline{q^n_{i}}-\overline{q^n_{i-1}}\right) \right], 
	\end{equation*}
	that is
	\begin{align}\overline{\left( \Delta\alpha \Delta\right)_{i}} = &  \frac{1}{\Delta x^2} \left\lbrace 
		\left[ \alpha^n_{i+\frac{1}{2}} + {\color{black} \frac{\Delta t}{2}} \partial_t \alpha_{i+\frac{1}{2}}\right]  \left[ 
		q_{i+1}^n-q_{i}^n + {\color{black} \frac{\Delta t}{2}}  \left(-\lambda \left(\Delta_{i+1}-\Delta_{i}\right)  \right. \right.\right. \notag\\
		& \left. \left. \left. + \left(\Delta\alpha\Delta\right)_{i+1} - \left(\Delta\alpha\Delta\right)_{i} + \beta\left(q_{i+1}^{n}-q_{i}^{n}\right)\right) 
		\right] 
		\right. \notag\\ &\left.
		+\left[ \alpha^n_{i-\frac{1}{2}} + {\color{black} \frac{\Delta t}{2}}  \partial_t \alpha_{i-\frac{1}{2}} \right] \left[
		q_{i-1}^n-q_{i}^n + {\color{black} \frac{\Delta t}{2}} \left(-\lambda \left(\Delta_{i-1}-\Delta_{i}\right)  \right. \right.\right. \notag\\
		& \left. \left. \left. + \left(\Delta\alpha\Delta\right)_{i-1} - \left(\Delta\alpha\Delta\right)_{i} + \beta\left(q_{i-1}^{n}-q_{i}^{n}\right)\right) 
		\right] 
		\right\rbrace  \label{eq:diffusion_adre}
	\end{align}
	{\color{black} where the time derivative of the viscous coefficient can be computed as
		\begin{equation} \partial_t \alpha_{i+\frac{1}{2}} \approx \frac{1}{2}\left( \frac{\alpha^n_{i+1}-\alpha^{n-1}_{i+1}}{\Delta t} + \frac{\alpha^n_{i}-\alpha^{n-1}_{i}}{\Delta t }\right).\end{equation}
		If $\alpha$ is a function depending on conservative variables (which would be the case, for example, of the viscous term for turbulent Navier-Stokes equations), we can also use the Cauchy-Kovalevskaya procedure.
	}

	% % % % % % % % % % % % % % % % % % % % % % % % % % % % % %
	%                    STEP 4                               %
	% % % % % % % % % % % % % % % % % % % % % % % % % % % % % %
	\subsubsection{Step 3.2. Numerical source}
	Lumping together the contributions of the reactive and diffusion terms we get a {\it numerical source} term as follows
	\begin{equation}s_i=\beta \overline{q_i} = \beta\left[ q_i+\frac{\Delta t}{2}\left(-\lambda \Delta_i +\left( \Delta\alpha \Delta\right)_i+\beta q_i\right) \right]. \label{eq:source_adre}\end{equation}
	Gathering \eqref{eq:flux_adre}, \eqref{eq:diffusion_adre} and \eqref{eq:source_adre}, the numerical scheme for the advection-diffusion-reaction equation, with $\lambda>0$, reads
	\begin{align}
		q_i^{n+1}=& q_{i}^{n} - \frac{\lambda \Delta t}{\Delta x}\left\lbrace
		q_{i}^{n} +\frac{1}{2} \Delta x\Delta_i + {\color{black}\frac{\Delta t}{2}} \left[-\lambda \Delta_{i} +\beta \left( q_{i}^{n} +\frac{1}{2} \Delta x\Delta_i \right) +\left( \Delta\alpha\Delta\right)_{i}\right]
		\right. \notag\\ &\left.
		-q_{i-1}^{n} -\frac{1}{2} \Delta x\Delta_{i-1} -  {\color{black}\frac{\Delta t}{2}}  \left[-\lambda \Delta_{i-1} +\beta \left( q_{i-1}^{n} +\frac{1}{2} \Delta x\Delta_{i-1} \right) +\left( \Delta\alpha\Delta\right)_{i-1}\right]
		\right\rbrace \notag\\
		& {\color{black} +}\frac{\Delta t}{\Delta x^2} \left\lbrace
		\left[ \alpha^n_{i+\frac{1}{2}} + \frac{\Delta t}{2}\partial_t \alpha_{i+\frac{1}{2}}\right]  \left[ 
		q_{i+1}^n-q_{i}^n + \frac{\Delta t}{2}\left(-\lambda \left(\Delta_{i+1}-\Delta_{i}\right)  \right. \right.\right. \notag\\
		& \left. \left. \left. + \left(\Delta\alpha\Delta\right)_{i+1} - \left(\Delta\alpha\Delta\right)_{i} + \beta\left(q_{i+1}^{n}-q_{i}^{n}\right)\right) 
		\right] 
		\right. \notag\\ &\left.
		+\left[ \alpha^n_{i-\frac{1}{2}} + \frac{\Delta t}{2}\partial_t \alpha_{i-\frac{1}{2}} \right] \left[
		q_{i-1}^n-q_{i}^n + \frac{\Delta t}{2}\left(-\lambda \left(\Delta_{i-1}-\Delta_{i}\right)  \right. \right.\right. \notag\\
		& \left. \left. \left. + \left(\Delta\alpha\Delta\right)_{i-1} - \left(\Delta\alpha\Delta\right)_{i} + \beta\left(q_{i-1}^{n}-q_{i}^{n}\right)\right) 
		\right] \right\rbrace \notag\\
		& + \beta\Delta t \left[ q_i+\frac{\Delta t}{2}\left(-\lambda \Delta_i +\left( \Delta\alpha \Delta\right)_i+\beta q_i\right) \right]. \label{eq:adre_num_scheme}
	\end{align}
	Finally, taking into account the centred approach of the slopes, \eqref{eq:centred_slopes} and \eqref{eq:centred_2slopes}, one obtains
	\begin{align} 
		q_i^{n+1}=& q_{i}^{n} - c\left\lbrace
		\frac{2+r}{2}\left( q_{i}^{n}-q_{i-1}^{n}\right)  +\frac{2-2c+r}{8} \left( q_{i+1}^{{\color{black} n}}- q_{i}^{{\color{black} n}}-q_{i-1}^{{\color{black} n}}+q_{i-2}^{{\color{black} n}}\right)   
		\right. \notag\\ &\left. +\frac{\Delta t}{2 \Delta x^2}\left[ \alpha^{{\color{black} n}}_{i+\frac{1}{2}}\left(q^{{\color{black} n}}_{i+1}-q_{i}^{{\color{black} n}}\right)-2\alpha_{i-\frac{1}{2}}^{{\color{black} n}}\left(q_{i}^{{\color{black} n}}-q_{i-1}^{{\color{black} n}}\right) +\alpha_{i-\frac{3}{2}}^{{\color{black} n}}\left(q_{i-1}^{{\color{black} n}}-q_{i-2}^{{\color{black} n}}\right)  \right]
		\right\rbrace \notag\\
		& {\color{black} +} \frac{\Delta t}{\Delta x^2} \left\lbrace
		\left[ \alpha^n_{i+\frac{1}{2}} + \frac{\Delta t}{2}\partial_t \alpha_{i+\frac{1}{2}}\right]  \left[ 
		\frac{2+r}{2}\left( q_{i+1}^n-q_{i}^n\right)  -\frac{c}{4} \left(q^{{\color{black} n}}_{i+2}-q_{i+1}^{{\color{black} n}}  \right. \right. \right. \notag\\
		& \left. \left.  \left.-q^{{\color{black} n}}_{i}+q^{{\color{black} n}}_{i-1}\right) + \frac{\Delta t}{2\Delta x^2}\left( \alpha_{i+\frac{3}{2}}^{{\color{black} n}}\left(q_{i+2}^{{\color{black} n}}-q_{i+1}^{{\color{black} n}}\right)-2\alpha_{i+\frac{1}{2}}^{{\color{black} n}}\left(q_{i+1}^{{\color{black} n}}-q_{i}^{{\color{black} n}}\right)  \right.  
		\right. \right.\notag\\ &  \left. \left. \left. +\alpha_{i-\frac{1}{2}}^{{\color{black} n}}\left(q_{i}^{{\color{black} n}}-q_{i-1}^{{\color{black} n}}\right)  \right)
		\phantom{\frac{b}{b}}\!\!\! \right] +\left[ \alpha^n_{i-\frac{1}{2}} + \frac{\Delta t}{2}\partial_t \alpha_{i-\frac{1}{2}} \right] \left[
		\frac{2+r}{2}\left(q_{i-1}^n-q_{i}^n\right)
		\right. \right. \notag\\ &\left. \left.
		-\frac{c}{4} \left(q^{{\color{black} n}}_{i}-q_{i-2}^{{\color{black} n}}-q_{i+1}^{{\color{black} n}}+q_{i-1}^{{\color{black} n}}\right)  + \frac{\Delta t}{2\Delta x^2}\left[ 2\alpha_{i-\frac{1}{2}}^{{\color{black} n}}\left(q_{i}^{{\color{black} n}}-q_{i-1}^{{\color{black} n}}\right) \right.\right.\right. \notag\\
		&  \left. \left. \left. -\alpha_{i-\frac{3}{2}}^{{\color{black} n}}\left(q^{{\color{black} n}}_{i-1}-q_{i-2}^{{\color{black} n}}\right) -\alpha_{i+\frac{1}{2}}^{{\color{black} n}}\left(q_{i+1}^{{\color{black} n}}-q_{i}^{{\color{black} n}}\right) \right] \phantom{\frac{b}{b}}\!\!\!
		\right] \right\rbrace \notag\\
		& + r\left[ q_i^{{\color{black} n}}+\left(-\frac{c}{4} \left( q_{i+1}^{{\color{black} n}}-q_{i-1}^{{\color{black} n}}  \right) +\frac{\Delta t }{2\Delta x^2}\left[ \alpha_{i+\frac{1}{2}}^{{\color{black} n}}\left(q_{i+1}^{{\color{black} n}}-q_{i}^{{\color{black} n}}\right)\right. \right.\right. \notag\\
		& \left. \left. \left.-\alpha_{i-\frac{1}{2}}^{{\color{black} n}}\left(q_{i}^{{\color{black} n}}-q_{i-1}^{{\color{black} n}}\right)  \right]+\frac{r}{2} q_i^{{\color{black} n}}\right) \right]. \label{eq:adre_num_scheme_complete}
	\end{align}
	where $c=\frac{\lambda\Delta t}{\Delta x}$ denotes the Courant number and $r=\beta\Delta t$ is called the reaction number.

	\begin{remark}[Constant diffusion coefficient]
		If we consider a constant diffusion coefficient  {\color{black} and denote $\alpha \Delta_i^{(2)}=(\Delta \alpha \Delta)_i$}, scheme \eqref{eq:adre_num_scheme} reads
		\begin{align}
			q_i^{n+1} = & q_i^n - \frac{\lambda\Delta t}{\Delta x}
			\left\lbrace q_{i}^{n} +\frac{1}{2} \Delta x\Delta_i + {\color{black} \frac{\Delta t}{2}}  \left[-\lambda \Delta_{i} +\beta \left( q_{i}^{n} +\frac{1}{2} \Delta x\Delta_i \right) +\alpha \Delta_{i}^{(2)}\right] \right. \notag\\
			& \left. 
			- q_{i-1}^{n} +\frac{1}{2} \Delta x\Delta_{i-1} -{\color{black} \frac{\Delta t}{2}}  \left[-\lambda \Delta_{i-1} +\beta \left( q_{i-1}^{n} -\frac{1}{2} \Delta x\Delta_{i-1} \right) +\alpha \Delta_{i-1}^{(2)}\right]
			\right\rbrace \notag\\
			&  + \frac{\alpha \Delta t}{\Delta x^2} 
			\left\lbrace \left(q_{i+1}^{n}-2q_{i}^{n}+q_{i-1}^{n} \right) +{\color{black} \frac{\Delta t}{2}} \left[  -\lambda \left(\Delta_{i+1}-2\Delta_{i}+\Delta_{i-1} \right)  
			\right.\right.\notag\\
			& \left.\left. +\alpha \left(\Delta^{(2)}_{i+1}-2\Delta^{(2)}_{i}+\Delta^{(2)}_{i-1} \right)  +\beta \left(q_{i+1}^{n}-2q_{i}^{n}+q_{i-1}^{n} \right)\right] \right\rbrace \notag\\
			& + \beta \Delta t \left[q_i^{n}+{\color{black} \frac{\Delta t}{2}} \left(-\lambda \Delta_i +\alpha \Delta_i^{(2)}+\beta q_i^{n}\right) \right]. \label{eq:ladre_num_scheme}
		\end{align}
		 Hence, the scheme for the advection-diffusion-reaction equation with constant diffusion coefficient becomes
		\begin{align} 
			q_i^{n+1}=& q_{i}^{n} - c\left\lbrace
			\frac{2+r}{2}\left( q_{i}^{n}  -q_{i-1}^{n}\right) +\frac{2-2c+r}{8} \left( q^{n}_{i+1}-q^{n}_{i-1}-q^{n}_{i}+q^{n}_{i-2}\right)   
			\right. \notag\\ &\left.
			+\frac{d}{2}\left[ q^{n}_{i+1}-3q^{n}_{i}+3q^{n}_{i-1}-q^{n}_{i-2} \right]
			\right\rbrace \notag\\
			& {\color{black} +}d\left\lbrace
			q_{i+1}^n-2q_{i}^n +q_{i-1}^n -\frac{c}{4} \left(q^{n}_{i+2}-2q^{n}_{i+1}  +2q^{n}_{i-1}-q^{n}_{i-2}\right)  \right. \notag\\
			&  \left.  + \frac{d}{2}\left[ q^{n}_{i+2}-4q^{n}_{i+1}+6q^{n}_{i}-4q^{n}_{i-1} +q^{n}_{i-2}\right] + \frac{r}{2}\left(q^{n}_{i+1}-2q^{n}_{i}+q_{i-1}^{n}\right)
			\right\rbrace \notag\\
			& + r\left[ q^{n}_i-\frac{c}{4} \left( q^{n}_{i+1}-q^{n}_{i-1}  \right) +\frac{d}{2}\left[ q^{n}_{i+1}-2q^{n}_{i}+q^{n}_{i-1}  \right]+\frac{r}{2} q^{n}_i\right]. \label{eq:ladre_num_scheme_complete}
		\end{align}
		where $d=\frac{\alpha \Delta t}{\Delta x^2}$.
	\end{remark}

	\begin{remark}[Advection-reaction equation]
		Assuming zero diffusivity, the scheme for the linear advection-reaction equation is recovered from \eqref{eq:adre_num_scheme},
		\begin{align}
			q_i^{n+1} = & q_i^n - \frac{\lambda\Delta t}{\Delta x}
			\left\lbrace q_{i}^{n} +\frac{1}{2} \Delta x\Delta_i + {\color{black} \frac{\Delta t}{2}}  \left[-\lambda \Delta_{i} +\beta \left( q_{i}^{n} +\frac{1}{2} \Delta x\Delta_i \right) \right] \right. \notag\\
			& \left. 
			- q_{i-1}^{n} +\frac{1}{2} \Delta x\Delta_{i-1} - {\color{black} \frac{\Delta t}{2}}  \left[-\lambda \Delta_{i-1} +\beta \left( q_{i-1}^{n} -\frac{1}{2} \Delta x\Delta_{i-1} \right) \right]
			\right\rbrace \notag\\
			& + \beta \Delta t \left[q_i^{n}+\frac{\Delta t}{2}\left(-\lambda \Delta_i +\beta q_i^{n}\right) \right].\label{eq:lare_num_scheme}
		\end{align}
		Furthermore, using centred slopes we get
		\begin{align} 
			q_i^{n+1}=& q_{i}^{n} - c \left[
			\frac{2+r}{2}\left( q_{i}^{n}-q_{i-1}^{n}\right)  +\frac{2-2c+r}{8} \left( q^{n}_{i+1}- q^{n}_{i}-q^{n}_{i-1}+q^{n}_{i-2}\right)    \right]\notag\\
			& + r\left[ q^{n}_i-\frac{c}{4} \left( q^{n}_{i+1}-q^{n}_{i-1}  \right) +\frac{r}{2} q^{n}_i \right]. \label{eq:lare_num_scheme_complete}
		\end{align}
	\end{remark}
	
	\begin{teorema}
		Schemes {\color{black} \eqref{eq:adre_num_scheme_complete}}, \eqref{eq:ladre_num_scheme_complete} and {\color{black}\eqref{eq:lare_num_scheme_complete}} are second order in space and time.
	\end{teorema}
	\begin{proof}
		The detailed proof is included in \ref{sec:appendix}.
	\end{proof}
	
	\begin{remark}
		In case the reconstruction done in Step 1 is done with constant polynomials and the half in time evolution of the variables given by the Taylor series expansion is neglected, the resulting scheme reduces to
		\begin{align}
			q_i^{n+1}=& q_{i}^{n} - c\left( q_{i}^{n} -q_{i-1}^{n} \right)  -\frac{\Delta t}{\Delta x^2} \left[
			\alpha^n_{i+\frac{1}{2}}  \left(
			q_{i+1}^n-q_{i}^n\right)
			+\alpha^n_{i-\frac{1}{2}}\left( q_{i-1}^n-q_{i}^n \right)  \right]  + r  q_i^n. \label{eq:adre_o1}
		\end{align}
		which is a first order in time and space scheme for the advection-diffusion-reaction equation \eqref{eq:adre}.
	\end{remark}

% % % % % % % % % % % % % % % % % % % % % % % % % % % % % %
% % % % % % % % % % % % % % % % % % % % % % % % % % % % % %
%                 MUSCL-Hancock                           %
% % % % % % % % % % % % % % % % % % % % % % % % % % % % % %
\section{MUSCL-Hancock}\label{sec:muscl_hancock}
The MUSCL-Hancock method, originally credited to Hancock in \cite{VL97}, is extended here to account for the source and diffusion terms. The extension is motivated by the ADER framework introduced earlier. First recall that the MUSCL-Hancock method for the homogeneous linear advection equation has the following steps:
\begin{description}
	\item[Step 1.] Data reconstruction.	
	First-degree polynomial for a cell $i$ are used, namely 
	\[p_i(x)=q_{i}^{n}+\Delta_i(x-x_i) .\]    
	
	\item[Step 2.] Computation of boundary extrapolated values.	      
	Cell boundary values are computed by simply evaluating the polynomials appropriately
	\[q_{i}^{L}=p_i(x_{i-\frac{1}{2}})=q_{i}^{n}-\frac{1}{2}\Delta x\Delta_i,\]
	\[q_{i}^{R}=p_i(x_{i+\frac{1}{2}})=q_{i}^{n}+\frac{1}{2}\Delta x\Delta_i.\]
	
	\item[Step 3.] Evolution of boundary extrapolated values.	
	Boundary-extrapolated values are evolved by half a time step,
	\begin{eqnarray*}
		\bar{q}^R_i= q^R_i-\frac{\Delta t}{2\Delta x} \left(f\left( q^R_i\right) -f\left( q^L_i\right)  \right), \\
		\bar{q}^L_i= q^L_i-\frac{\Delta t}{2\Delta x} \left(f\left( q^R_i\right) -f\left( q^L_i\right)  \right).
	\end{eqnarray*}     
	
	\item[Step 4.] Solution of the Riemann problem and numerical flux.	
	The evolved boundary-extrapolated values are used to define a classical Riemann problem at each 
	intercell boundary,
	\begin{equation*}
		\left\lbrace\begin{array}{l}
			\partial_t q \left(x,t \right) +\lambda\partial_x q\left(x,t \right)=0,\\
			q(x,0)=\left\lbrace \begin{array}{ll}
				\bar{q}_{i}^{R} & x<0, \\
				\bar{q}_{i+1}^{L} & x>0, 
			\end{array}\right. \end{array}\right.
	\end{equation*}
	the solution of which is
	\begin{equation*}
		q(x,t)=\left\lbrace\begin{array}{ll}
			\bar{q}_i^R& \dfrac{x}{t}<\lambda,\\
			\bar{q}_{i+1}^L& \dfrac{x}{t}>\lambda.
		\end{array} \right.
	\end{equation*}
	Hence, the sought intercell flux is given by
	\begin{equation*}
		f^{MH}_{i+\frac{1}{2}}= %f\left( q_{i+\frac{1}{2}}(0)\right)=
		\left\lbrace\begin{array}{lr}
			\lambda \bar{q}_i^R=\lambda \left(q^n_i + \frac{1-c}{2}\Delta x\Delta_i  \right)  &  \lambda>0,\\
			\lambda \bar{q}_{i+1}^L=\lambda \left(q^n_{i+1}-\frac{1+c}{2} \Delta x\Delta_{i+1} \right) &  \lambda<0.
		\end{array} \right.
	\end{equation*}    	     
	Note that if no reconstruction is performed, the MUSCL-Hancock method reduces to the Godunov first-order method, with the particular numerical flux employed in the last step.
	
	By choosing centred slopes, as already done for ADER, and assuming $\lambda>0$ (the discussion of the case $\lambda<0$ is analogous) we obtain the MUSCL-Hancock scheme for the linear advection equation:
	\begin{equation}
		q_{j}^{n+1} =   q_{j}^{n}-c\left[q_{j}^{n}-q_{j-1}^{n}+\frac{1-c}{4}\left(q_{j+1}^{n}-q_{j}^{n}-q_{j-1}^{n}+q_{j-2}^{n} \right)\right].\label{eq:lae_num_scheme}
	\end{equation}
\end{description}

% % % % % % % % % % % % % % % % % % % % % % % % % % % % % % % % % % % %
\subsection{Source and diffusion terms}
The inclusion of reaction and diffusion terms is accomplished by modifying Step 3, in which such terms at the half time are added to {\color{black} the} evolution of boundary extrapolated values. Thus we obtain:
\begin{gather}
	\bar{q}^R_i= q^R_i-\frac{\Delta t}{2}\left\lbrace \frac{1}{\Delta x} \left(f\left( q^R_i\right) -f\left( q^L_i\right)  \right) 
	\right. \notag\\ \left.
	- \frac{1}{\Delta x^2} {\color{black} g\left(\left(\alpha \Delta q\right)^R_i,\left(\alpha \Delta q\right)^L_i \right) } -
	s\left( q^R_i\right)\right\rbrace\\
	\bar{q}^L_i= q^L_{i+1}-\frac{\Delta t}{2}\left\lbrace \frac{1}{\Delta x}\left(f\left( q^R_{i+1}\right) -f\left( q^L_{i+1}\right) \right)
	\right. \notag\\ \left.
	- \frac{1}{\Delta x^2}{\color{black} g\left(\left(\alpha \Delta q\right)^R_{i+1},\left(\alpha \Delta q\right)^L_{i+1} \right)
	}
	-s\left( q^L_{i+1}\right) 
	\right\rbrace
\end{gather}
with
\[g\left(\left(\alpha \Delta q\right)^R_i,\left(\alpha \Delta q\right)^L_i \right)= \alpha^{n}_{i+\frac{1}{2}} \left(q^{n}_{i+1}-q^{n}_{i}\right) -  \alpha^{n}_{i-\frac{1}{2}} \left(q^{n}_{i}-q^{n}_{i-1}\right). \]
The final step is as before, that is, the numerical flux is computed by solving the Riemann problem for the linear advection equation with evolved boundary-extrapolated values as initial conditions. Just as ADER, the numerical flux includes the contribution due to diffusion and source terms. Additional contributions to the scheme resulting from diffusion and reaction are accounted for by following the ADER approach introduced in Section \ref{sec:ader}.

\begin{remark}
	The resulting schemes for the linear advection-diffusion-reaction equation, constructed from the ADER and MUSCL-Hancock approaches, are algebraically identical.
\end{remark}

% % % % % % % % % % % % % % % % % % % % % % % % % % % % % %
% % % % % % % % % % % % % % % % % % % % % % % % % % % % % %
%                   Stability                             %
% % % % % % % % % % % % % % % % % % % % % % % % % % % % % %
\section{Stability analysis}\label{sec:stability}
The stability analysis of the obtained schemes is divided into two cases. On the one hand, linear advection equation allows for an easy computation of the stability region. On the other hand, advection-diffusion-reaction equation with constant diffusion coefficient will be analysed thanks to graphical representation.

\subsection{Linear advection equation}
Stability analysis of linear models is done following von Neumann stability analysis procedure, {\color{black} see \cite{VC15}, \cite{Str04}}. 
Let us consider the trial function \[q^{n}_i=A^{n}e^{I \theta i},\] where $A\in\mathbb{C}$ represents an amplitude, $I$ denotes the complex unity so that $i$ is kept for the mesh, and $\theta=P\Delta x$ is an angle with $P$ the wave number in the $x$-direction. Then, \eqref{eq:lae_num_scheme} yields to
\begin{gather*}
	A^{n+1}e^{I\theta i}= A^{n}e^{I\theta i}- c \left[ A^{n}e^{I\theta i}- A^{n}e^{I\theta \left( i-1\right) } + \frac{1-c}{4} \left(A^{n}e^{I\theta \left( i-1\right) } \right.\right.\\ \left.\left. -A^{n}e^{I\theta i }  -A^{n}e^{I\theta \left( i-1\right) } +A^{n}e^{I\theta \left( i-2\right) } \right)  \right],
\end{gather*}
\begin{equation*}
	A= 1- c \left[ 1- e^{-I\theta} + \frac{1-c}{4} \left(e^{-I\theta} -1 -e^{-I\theta } +e^{-2I\theta } \right)  \right],
\end{equation*}
hence,
\begin{equation*}
	\left\| A\right\|^2  
	= c \left( c-1\right) \left(\cos\theta-1 \right)^2 \left[
	\frac{1}{2}c\left( c-1\right) \left( \cos\theta+1\right) +1
	\right] +1 .
\end{equation*}
The stability condition, $\left\|A\right\|^2\leq 1$, is verified if and only if
\begin{gather*}
	c \left( c-1\right) \left(\cos\theta-1 \right)^2 \left[
	\frac{1}{2}c\left( c-1\right) \left( \cos\theta+1\right) +1 \right] \leq 0\\
	\Leftrightarrow 
	c\left( c-1\right) \left( \cos\theta+1\right)  \geq -2 \textrm{ and } c\leq 1.
\end{gather*}
From which it follows that the scheme is stable if the Courant number, $c$ lies between zero and unity; it is conditionally stable with stability condition \[0\leq c\leq 1.\]

Sometimes the amplification factor is a difficult expression to deal with. In order to make it easier, we can represent
the value of the function of the binomial expression of the amplification factor,
\renewcommand{\arraystretch}{1.2}
\begin{equation}
	\begin{array}{ccc}
		A:\left[-\pi,\pi\right] \times \mathbb{R} &\longrightarrow & \mathbb{C} \\ 
		(\theta,c) &\rightsquigarrow & A(\theta,c),
	\end{array}
\end{equation}
for different values of $c${\color{black},}  see \cite{Hoff89}. 
In Figure \ref{fig:mh_hoffman}, we can observe that the functions whose image is completely contained in the square $[-1,1]\times [-1,1] \subset \mathbb{C}$ are defined for $c\in[0,1]$. This agrees with the analytical results already obtained. 

\begin{figure}
	\centering
	\includegraphics[width=0.9\linewidth]{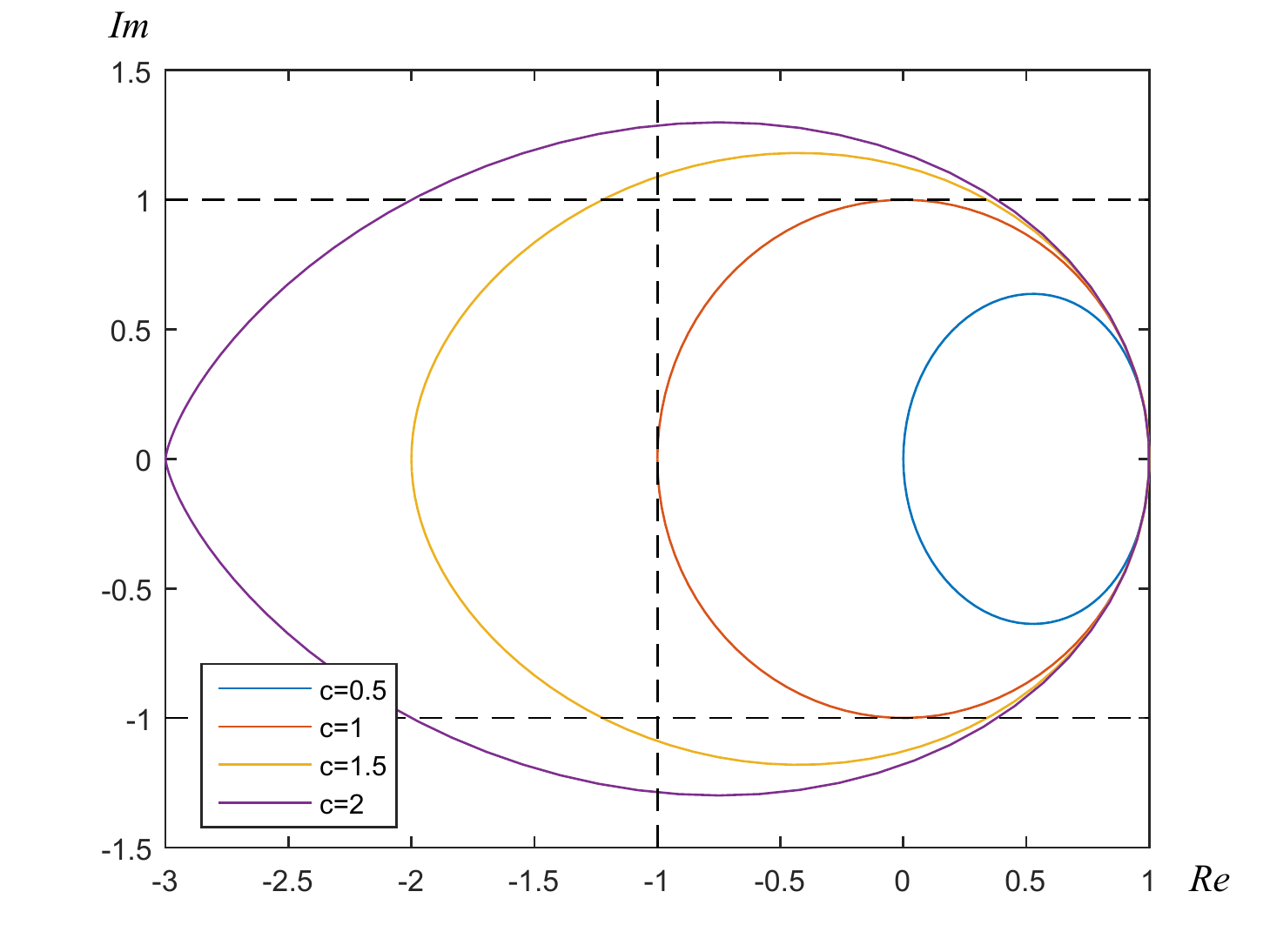}
	\caption{Representation in the complex plane of the amplification factor values for several {\color{black} values of} $c$ with $\theta\in \left[-\pi,\pi\right]$. {\color{black} It can be observed that $c$ is bounded above by 1 when $\left\|A \right\|$ is bounded by 1}.}
	\label{fig:mh_hoffman}
\end{figure}

\begin{figure}
	\centering
	\includegraphics[width=0.495\linewidth]{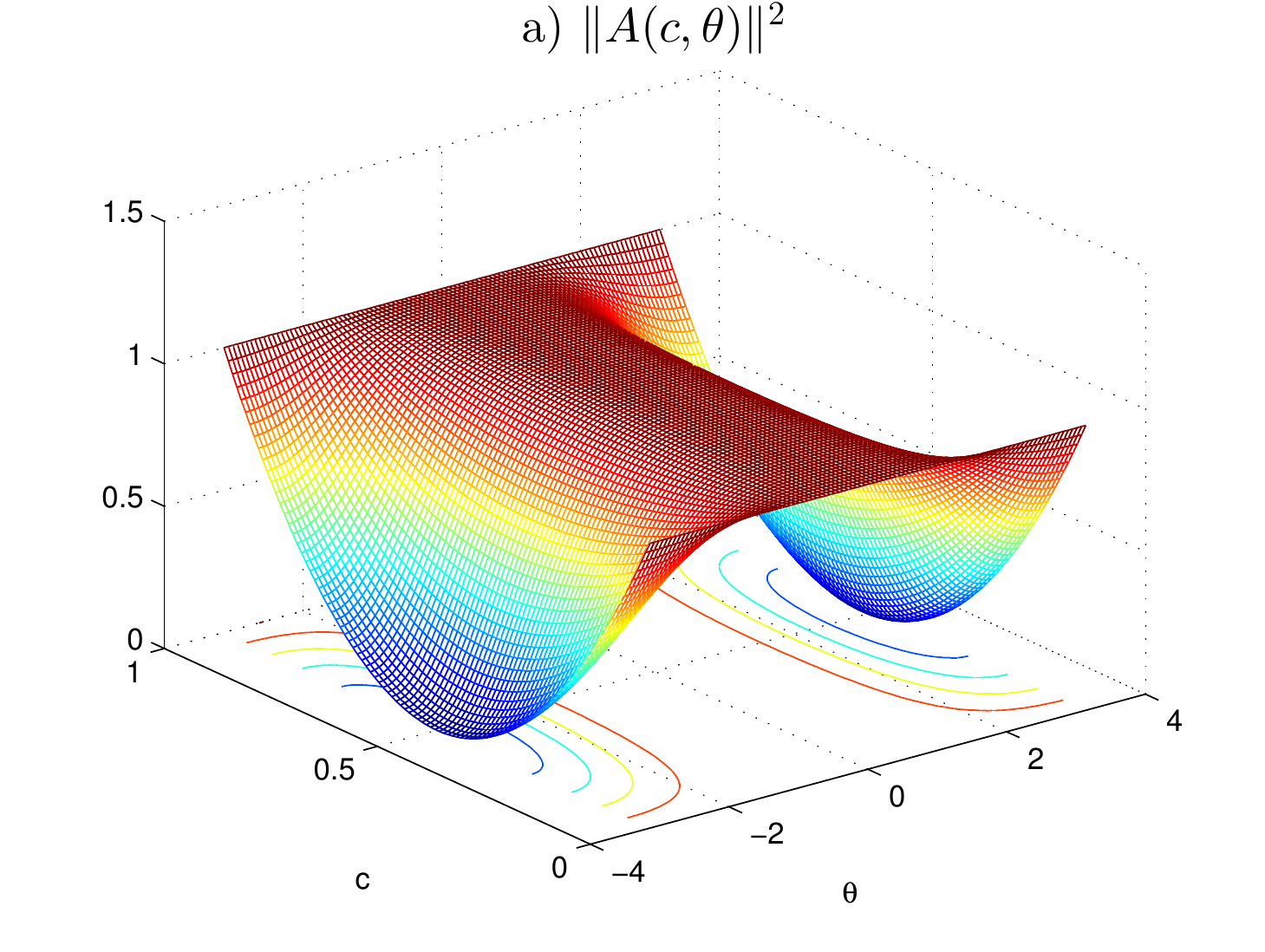}
	\hfill \includegraphics[width=0.495\linewidth]{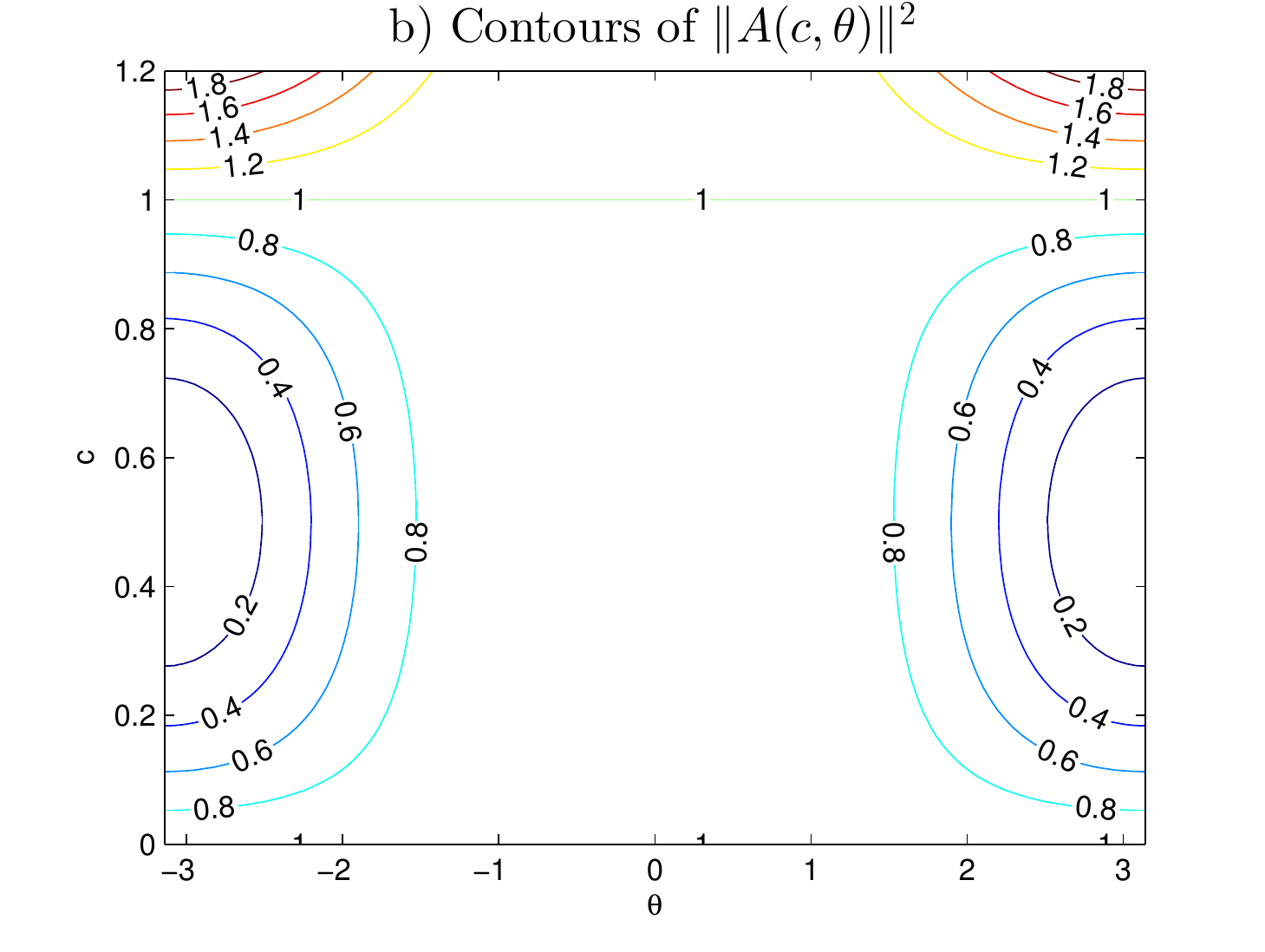}
	\caption{Graph and contour lines of $\left\|A(\theta,c) \right\|^2$ for scheme \eqref{eq:lae_num_scheme}. The stability region is clearly determined for $0\leq c\leq 1$.}
	\label{fig:mh_c_a2}
\end{figure}

On the other hand, we can plot the function defined by the norm of the amplification factor, $\left\| A(\theta,c)\right\|^2\in\mathbb{R}$.
As $A$ depends on two variables, $\theta$ and $c$, the plot, Figure \ref{fig:mh_c_a2}a, is a surface in $\mathbb{R}^3$.
Drawing the contour lines % of level one
we can confirm that the stability condition is verified if and only if  $0\leq c\leq 1$.
In Figure \ref{fig:mh_c_a2}b we consider $c\in [0,1.2]$ to remark that for any {\color{black} chosen} $c_{0}>1$ there exist $\theta_{c_{0}} \in \left[-\pi,\pi\right]$ such that the values of $\left\| A(\theta_{c_{0}},c_{0})\right\|$ are larger than one.

\subsection{Linear advection-diffusion-reaction equation}
The amplification factor of scheme \eqref{eq:ladre_num_scheme}, which depends on the angle $\theta$ and on the parameters $c=\frac{\lambda\Delta t}{\Delta x}$, $d=\frac{\alpha\Delta t}{\Delta x^2}$ and $r=\beta\Delta t$, is computed using the von Neumann procedure obtaining
\begin{align*}
	A= &1-c \left\lbrace 1- \cos \theta +I\sin\theta +\frac{1-c}{4} \left(2I\sin \theta -1 + \cos \left( 2\theta\right) -I\sin \left(2\theta \right)  \right)\right.\\
	& +\frac{r}{2}\left[1-\cos\theta +I\sin \theta +\frac{1}{4}\left(2I\sin\theta-1+ \cos \left( 2\theta\right) -I\sin \left(2\theta \right)\right) \right] \\
	& \left.  +\frac{d}{2}\left( 4\cos\theta -2I\sin\theta -3 -\cos\left( 2\theta-I\sin\left( 2\theta\right) \right) \right)\right\rbrace\\
	& +d\left\lbrace 2\cos\theta-2 -\frac{c}{4}\left(2I\sin\left( 2\theta\right) -6I\sin\theta \right) 
	+\frac{d}{2}\left(2\cos\left(2\theta\right)-8\cos\theta+6 \right)\right.\\
	&\left. +\frac{r}{2}\left( 2\cos\theta-2\right) \right\rbrace+r\left\lbrace 1-\frac{c}{2}I\sin\theta +d\left( \cos\theta -1\right) +\frac{r}{2}\right\rbrace.
\end{align*}

As the bounds of $c,\,d \, \textrm{and} \, r$ in order to limit the amplification factor are interdependent,
the computation of the constraints will produce {\color{black} complicated} expressions.
Still, a graphical representation provides us with a good approach to determine the stability region.

\vspace{0.5cm}
\noindent \textbf{Linear advection-reaction equation}

\vspace{0.2cm}
The function related with the amplification factor of the linear advection-reaction equation results
\begin{align*}
	A: [-\pi,\pi]\times\mathbb{R}\times\mathbb{R}\; &\longrightarrow \;\;\; \mathbb{R}\\
	(\theta,c,r)\; &\longrightarrow  \; A(\theta,c,r).
\end{align*}
Its graph is embedded in $\mathbb{R}^4$, therefore, instead of plotting contour lines, we represent the isosurface of level one
which splits $\mathbb{R}^3$ in two domains (see Figure \ref{fig:stability_c_r_d0b}).
One of them, which contains the point $(\theta,c,r)=(0,0,0)$, is the stability region of the scheme.
Inside the other domain the scheme is unconditionally unstable.
Furthermore, the orthogonal planes  to the $r$-axis, that is, the planes resulting for a fixed value of $r$, provide the contour plots of level one for the linear advection-reaction equation related to the set $r$ (see the two-dimensional subplots of Figure \ref{fig:stability_c_r_d0b} where $S$ denotes the stability region of the scheme).
\begin{figure}
	\centering
	\hspace*{-1cm}\includegraphics[width=1.15\linewidth]{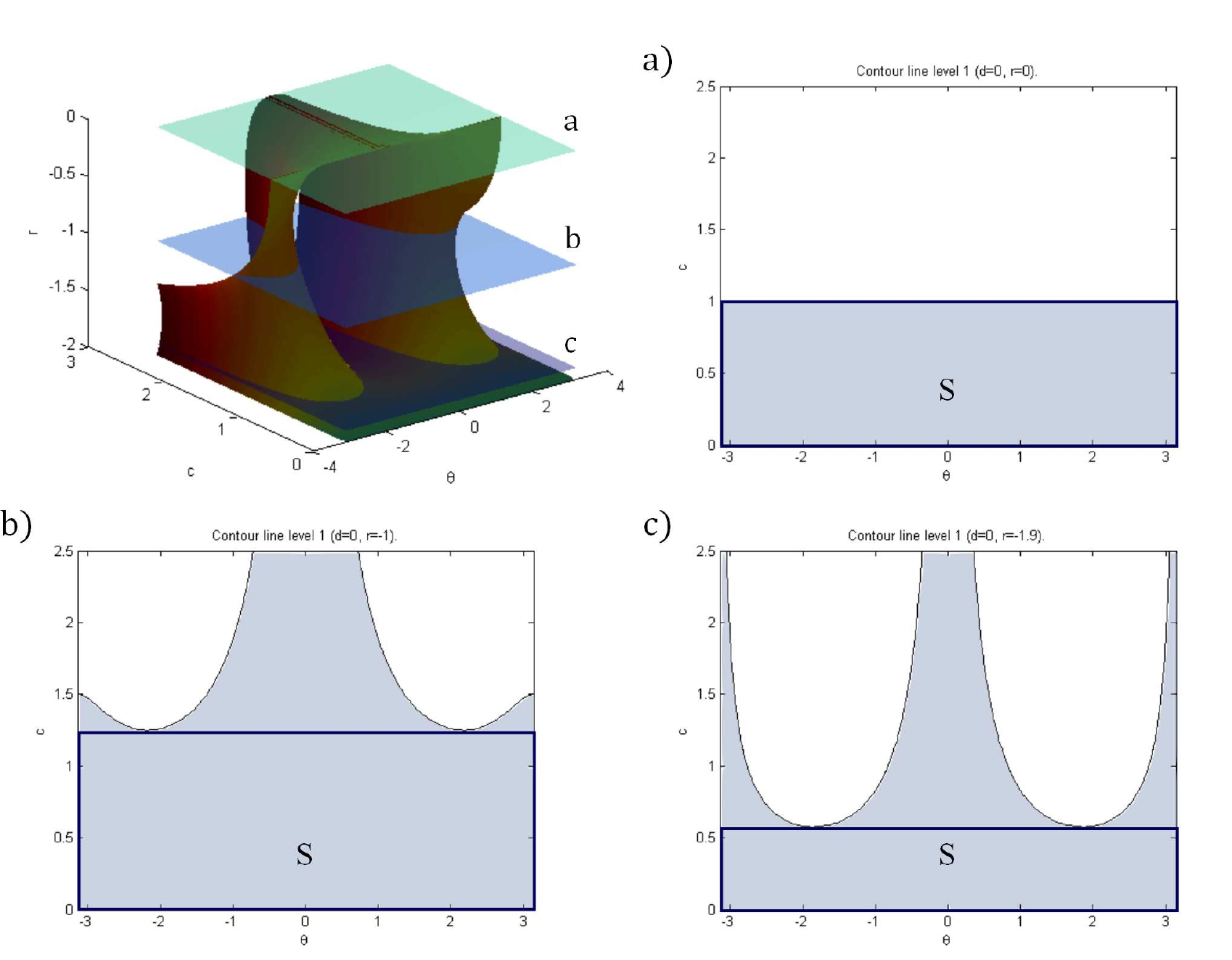}
	\caption{Stability region for the linear advection-reaction equation. The isosurface of level one splits $\mathbb{R}^3$ into the stability region, containing the origin, and the unstable region. Subplots a), b), c) represent the contour plot of level one for the fixed values of $r=0$, $r=-1$ and $r=-1.9$ respectively. The shaded regions correspond to the stability region. The blue rectangles identified as $S$ are the admissible regions of stability.}
	\label{fig:stability_c_r_d0b}
\end{figure}

For instance, assuming $r=-1$ the value $c=1.1$ guarantees the stability. 
However, we must carefully  analyse these results.
For a specific problem, with a given mesh, setting $r=-1$ does not imply that the $\Delta t$ is such that $c=1.1$ and vice versa.
A particular example will help us to understand {\color{black}the situation better}.
We consider the linear advection-reaction equation
\[\partial_t q(x,t)+\lambda_0\partial_x q(x,t)=\beta_0 q(x,t)\]
with fixed $\lambda_0\in \mathbb{R}^{+},\beta_0\in \mathbb{R}^{-}$.
If the mesh size is $\Delta x=\Delta x_0$, then it is verified
\begin{equation}
	\left\lbrace \begin{array}{l} r=\beta_0 \Delta t,\\[10pt]
		c=\dfrac{\lambda_0 \Delta t}{\Delta x_0}.\end{array} \right.
	\label{eq:r_c_relation}
\end{equation}
So, if $\Delta t$ is computed from $r=r_0$ fixed, then \begin{equation}c=\frac{\lambda_0 r_0}{\beta_0 \Delta x_0}\label{eq:c_relation_r}\end{equation} is determined.
Similarly, given $c=c_0$ the value of $\Delta t$ is resolved and \begin{equation}r=\frac{\beta_0c_0\Delta x_0}{\lambda_0}.\label{eq:r_relation_c}\end{equation}
Hence, for $r=-1$  the value of $c$ is determined and can be different from $1.1$. In case it is bigger, we would have fallen into the unstable region.

To avoid the previous trouble, we define rectangular cuboids
\begin{equation}O_{c,r}=\left\lbrace (\theta,c,r)\; |\; \theta\in[-\pi,\pi],\; c\in [0,c_{M}],\; r \in[r_{m},0], \; c_{M}\in\mathbb{R}^{+}, \; r_{m}\in\mathbb{R}^{-}\right\rbrace\label{eq:rc_cr}\end{equation}
embedded in the stability region.
Selecting $c_{M}=1$, the upper bound of $c$,  and $r_{m}=-1$, the lower bound of of $c$, the resulting scheme is stable.

\vspace{0.4cm}
\noindent \textbf{Linear advection-diffusion equation}

\vspace{0.2cm}
The previous procedure can also be applied for the linear advection-diffusion equation.
Hence, we consider
\begin{align*}
	A: [-\pi,\pi]\times\mathbb{R}\times\mathbb{R}\; &\longrightarrow \;\;\; \mathbb{R}\\
	(\theta,c,d)\; &\longrightarrow  \; A(\theta,c,d)
\end{align*}
and
\begin{equation}O_{c,d}=\left\lbrace (\theta,c,d)\; |\; \theta\in[-\pi,\pi],\; c\in [0,c_{M}],\; d \in[0,d_{M}], \; c_{M},\, d_{M}\in\mathbb{R}^{+}\right\rbrace.\label{eq:rc_cd}\end{equation}
In Figure \ref{fig:stability_c_r0_d_19b}, we can observe that $c_{M}=1$ and $d_{M}=0.5$ generate an admissible cuboid.

\vspace{0.4cm}
\noindent \textbf{Linear advection-diffusion-reaction equation}

\vspace{0.2cm}
As {\color{black}the} last step, we study the stability for the linear advection-diffusion-reaction equation. The amplification factor function reads
\begin{align*}
	A: [-\pi,\pi]\times\mathbb{R}\times\mathbb{R}\times\mathbb{R}\; &\longrightarrow \;\;\; \mathbb{R}\\
	(\theta,c,d,r)\; &\longrightarrow  \; A(\theta,c,d,r).
\end{align*}
so its graph belongs to $\mathbb{R}^5$ and the isosurfaces are embedded in  $\mathbb{R}^4$.
Admissible regions can be established through  4-orthotopes,
\begin{align} O_{c,d,r}=\left\lbrace (\theta,c,r,d)\; |\; \theta\in[-\pi,\pi],\; c\in [0,c_{M}],\; d \in[0,d_{M}], \; r \in[r_{m},0],\right.\\ \left. \; c_{M},\, d_{M}\in\mathbb{R}^{+}, \; r_{m}\in\mathbb{R}^{-}\right\rbrace. \label{eq:rc_cdr}\end{align}

To get an idea of the shape of the stability region, 
we can picture an evolutionary problem where one of the variables, for instance $r$, plays the role of the time and the remaining {\color{black}ones} are considered as spacial variables.  
Therefore, the stability region is determined by the intersection of the stability regions for the different snapshots of $r$.
Figures \ref{fig:stability_c_r0_d_19b}-\ref{fig:stability_c_r1_d_19b}  show the graphs obtained for fixed values of $r$.
We can conclude that $c_{M}=1$, $d_{M}=\frac{1}{4}$ and $r_{m}=-\frac{1}{2}$ define a 4-orthotope embedded in the stability region.

\begin{figure}
	\centering
	\hspace*{-1cm}\includegraphics[width=1.15\linewidth]{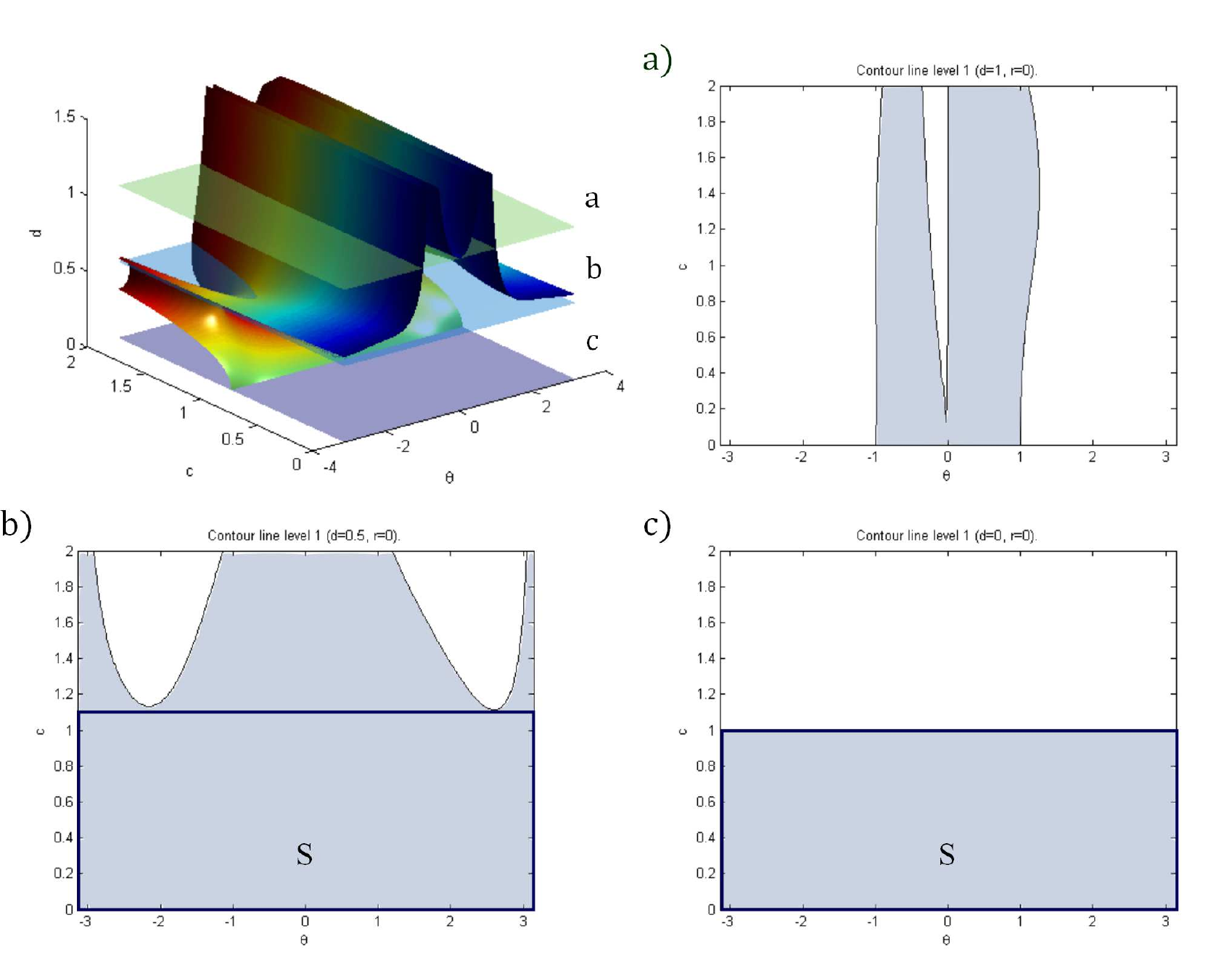}
	\caption{Stability region for the linear advection-diffusion equation. The isosurface of level one splits $\mathbb{R}^3$ into the stability region, containing the origin, and the unstable region. Subplots a), b), c) represent the contour plot of level one for the fixed values of {\color{black}$d=1$}, $d=0.5$ and {\color{black}$d=0$} respectively. The shaded regions correspond to the stability region. The blue rectangles identified as $S$ are the admissible regions of stability.}
	\label{fig:stability_c_r0_d_19b}
\end{figure}

\begin{figure}
	\centering
	\hspace*{-1cm}\includegraphics[width=1.15\linewidth]{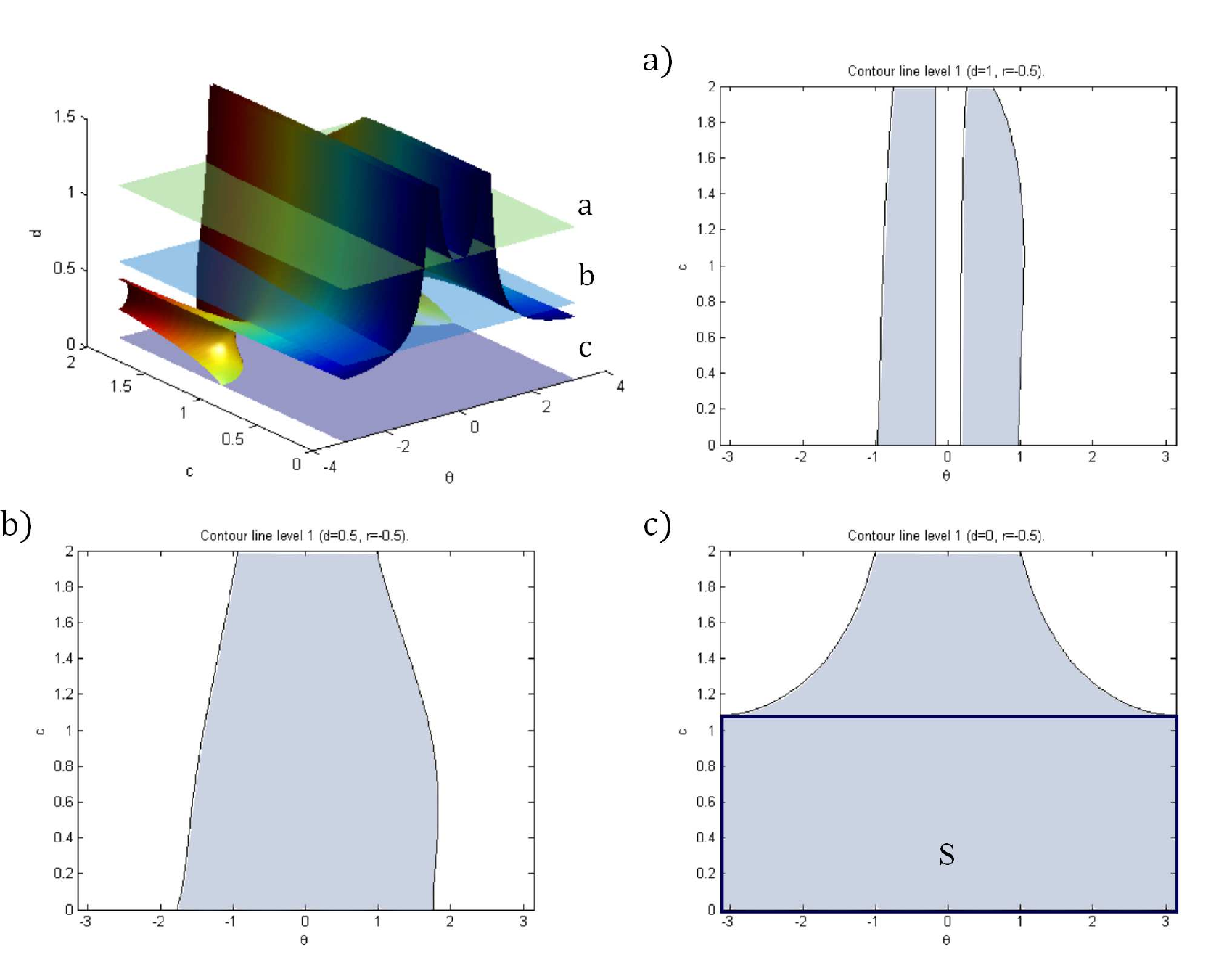}
	\caption{Stability region for the linear advection-diffusion-reaction equation with fixed reaction number $r=-0.5$. The isosurface of level one splits $\mathbb{R}^3$ into the stability region, containing the origin, and the unstable region. Subplots a), b), c) represent the contour plot of level one for the fixed values of {\color{black}$d=1$}, $d=0.5$ and {\color{black}$d=0$} respectively. The shaded regions correspond to the stability region. The blue rectangles identified as $S$ are the admissible regions of stability.}
	\label{fig:stability_c_r05_d_19b}
\end{figure}

\begin{figure}
	\centering
	\hspace*{-1cm}\includegraphics[width=1.15\linewidth]{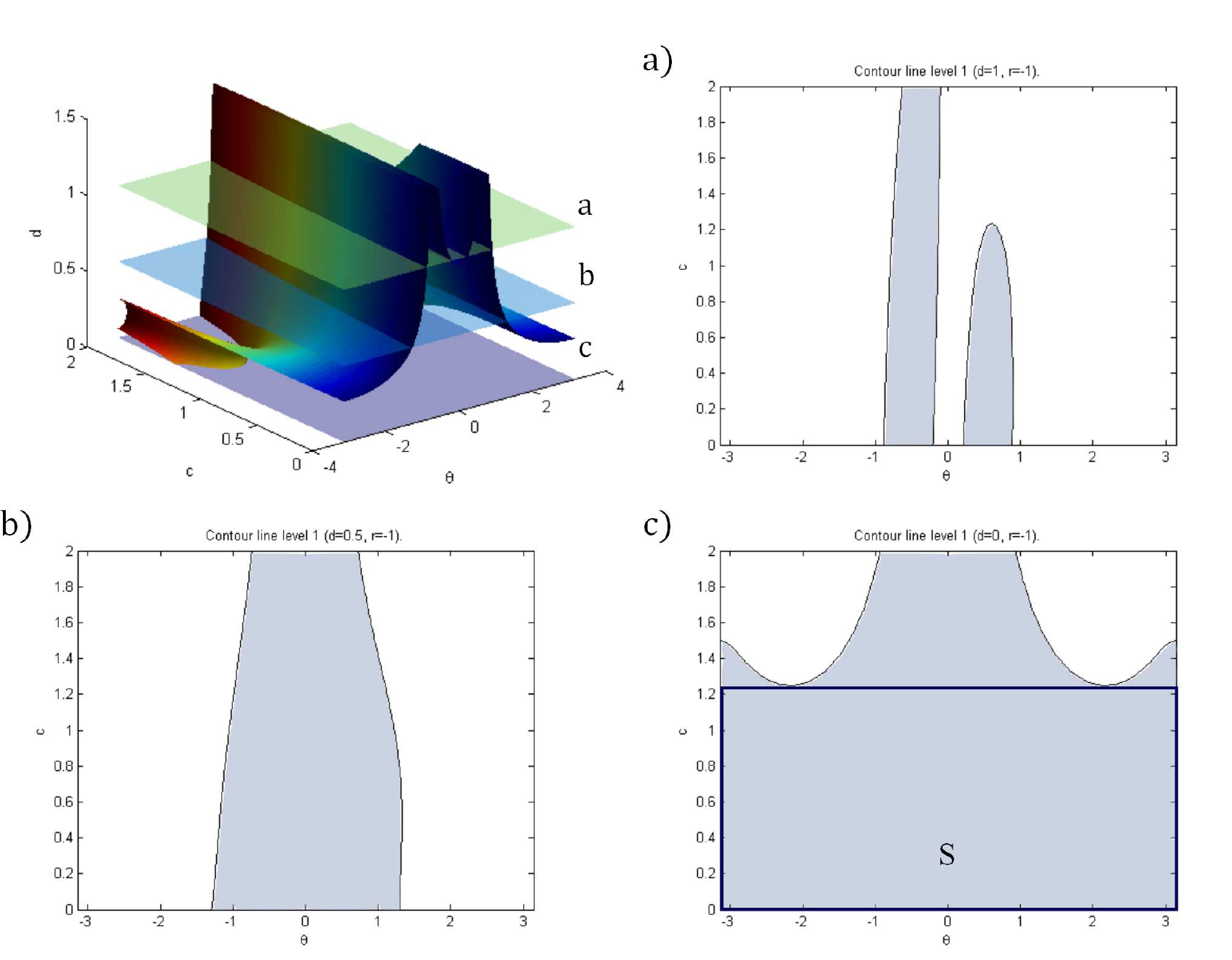}
	\caption{Stability region for the linear advection-diffusion-reaction equation with fixed reaction number $r=-1$. The isosurface of level one splits $\mathbb{R}^3$ into the stability region, containing the origin, and the unstable region. Subplots a), b), c) represent the contour plot of level one for the fixed values of {\color{black}$d=1$}, $d=0.5$ and {\color{black}$d=0$} respectively. The shaded regions correspond to the stability region. The blue rectangles identified as $S$ are the admissible regions of stability.}
	\label{fig:stability_c_r1_d_19b}
\end{figure}

{\color{black}  A new alternative way to depict the stability region is to plot the isosurface of level one of the function defined by
	\begin{equation}
		m_{\theta}(c,d,r)=\max_{\theta\in [-\pi,\pi]} \left\| A(\theta,c,d,r)\right\| .\label{eq:amplification_mtheta}
	\end{equation}
	Figure \ref{fig:stability_20202020}  confirms that the 4-orthotope defined above,  $\mathcal{O}_{1,\frac{1}{4},-\frac{1}{2}}$, is embedded in the stability region.}
\begin{figure}
	\centering
	\includegraphics[width=0.5\linewidth]{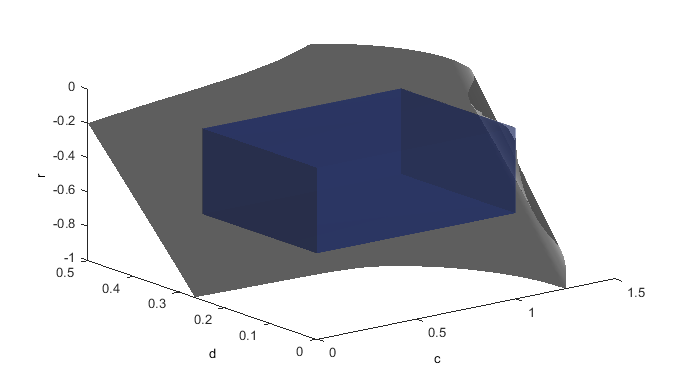}\hfill
	\includegraphics[width=0.5\linewidth]{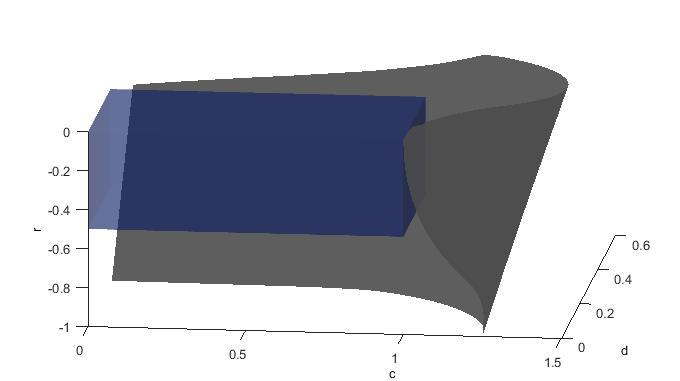}
	\caption{{\color{black}Two different views of the isosurface of level one of function $m_{\theta}$, \eqref{eq:amplification_mtheta}, (grey) and the 4-orthotope of stability $\mathcal{O}_{1,\frac{1}{4},-\frac{1}{2}}$ for the linear advection-diffusion-reaction equation (blue).}}
	\label{fig:stability_20202020}
\end{figure}

% % % % % % % % % % % % % % % % % % % % % % % % % % % % % %
% % % % % % % % % % % % % % % % % % % % % % % % % % % % % %
%               Numerical results                         %
% % % % % % % % % % % % % % % % % % % % % % % % % % % % % %
\section{Numerical results}\label{sec:numerical_results}
In this section, we present the results obtained for several test problems.
The error is analysed {\color{black}by} computing the norms
\begin{equation*}
	\textrm{Err}_{\mathcal{L}^1}=\left\|q-\hat{q} \right\|_{l^{\infty}\left( \mathcal{L}^1(\Omega)\right) },
	\, \textrm{Err}_{\mathcal{L}^2}=\left\|q-\hat{q} \right\|_{l^{\infty}\left( \mathcal{L}^2(\Omega)\right) },\,
	\textrm{Err}_{\mathcal{L}^\infty}=\left\|q-\hat{q} \right\|_{l^{\infty}\left( \mathcal{L}^\infty(\Omega)\right),}
\end{equation*}
where $\hat{q}$ denotes the numerical solution and $q$ is whether the exact solution or a reference solution computed for a refined mesh if the problem does not have an analytical solution.

% % % % % % % % % % % % % % % % % % % % % % % % % % % % % %
%                   TEST 1                                %
% % % % % % % % % % % % % % % % % % % % % % % % % % % % % %
\subsection{Test 1. Advection-reaction equation}
We consider two different tests for the advection-reaction equation.
{\color{black}  For both of them, Dirichlet boundary conditions are set. The exact solution is imposed at the boundary nodes and for the computation of the numerical flux {\color{black} at the first node} we use a forward approximation of the slope, namely,
	\begin{equation*} \Delta_1=\frac{q_2-q_1}{\Delta x}, \end{equation*}}
{\color{black} an analogous procedure is considered for the last node}.	

% % % % % % % % % % % % % % % % % % % % % % % % % % % % % %
%                   TEST 1.1                              %
% % % % % % % % % % % % % % % % % % % % % % % % % % % % % %
\subsubsection{Test 1.1.}
The first test problem studied is given by
\begin{eqnarray}
	\partial_t q (x,t) + \partial_x  q(x,t) = -q(x,t), \quad q(x,0)=\exp(-2x^2), \label{eq:test1}
\end{eqnarray}
with exact solution
\begin{equation*}
	q(x,t)=\exp(-2(x-\lambda t)^2+\beta t).
\end{equation*}

Seven meshes are considered. The time step is determined to guarantee that $c$ and $r$ belong to the rectangular cuboid \eqref{eq:rc_cdr} defined by $c_{M}=1,\, r_{m}=-1$. Since the time step condition imposed by $c_{M}$ is lower than the defined by  $r_{m}$, the values of $r$ are computed following \eqref{eq:r_relation_c}.

The obtained $r$, errors and order are depicted in Table \ref{tab:test_sol_exacta}. The {\color{black}attained}  second order {\color{black} was} theoretically {\color{black}expected}.
The results for the {\color{black}  mesh with 32 nodes} are depicted in Figure \ref{fig:test1_exp_t1}.

It is important to notice that neither of the chosen values for $c_{M}$ and $r_{m}$ are the optimal in the case of an advection-reaction equation. That is, within this test the exact solution is not expected to be obtained.

\begin{figure}
	\centering
	\includegraphics[width=\linewidth]{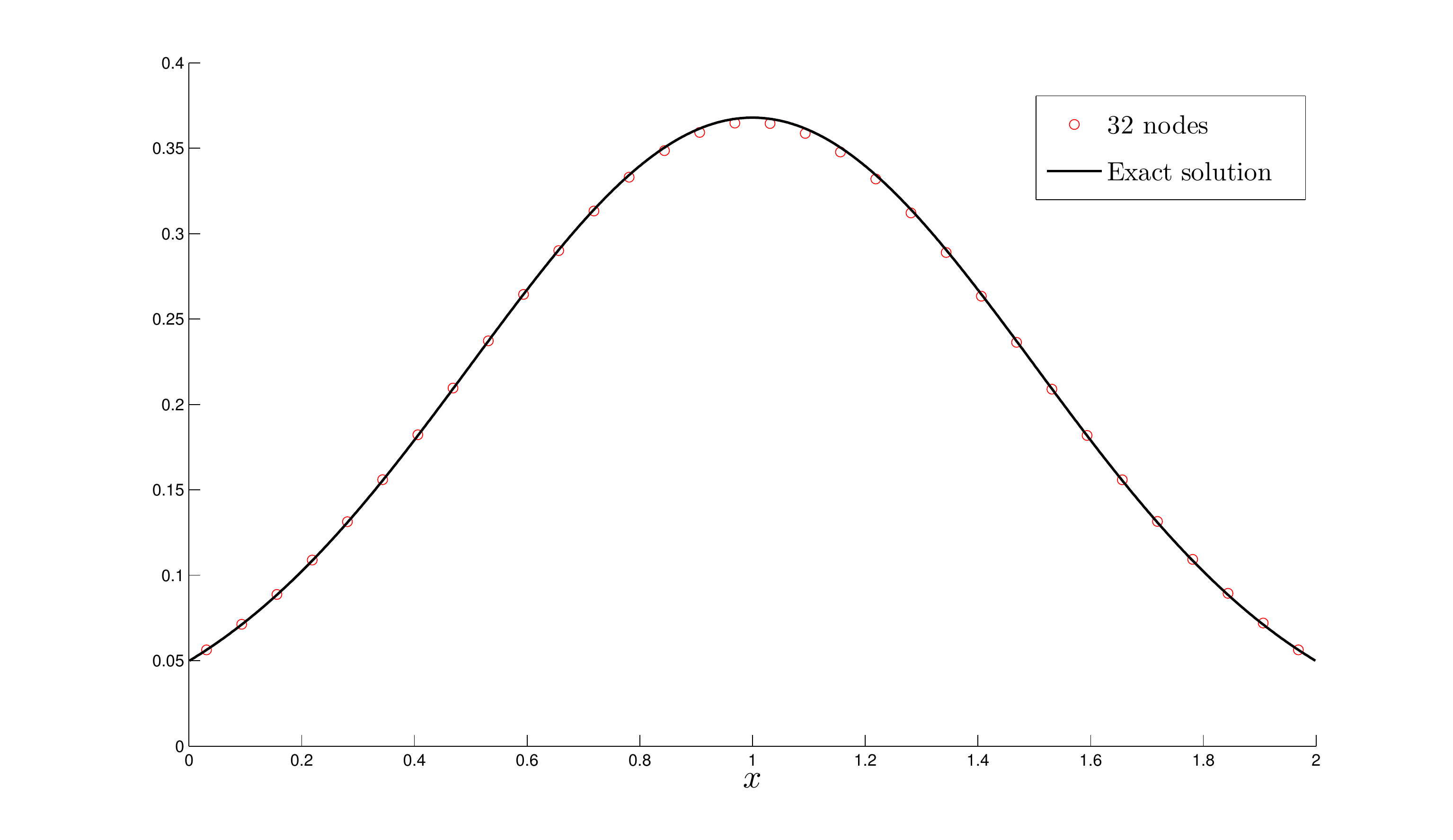}
	\caption{{\color{black} Test 1.1. Exact solution and numerical results obtained for the mesh with 32 nodes.
			$\Omega=[0,2]$, $t_{\textrm{end}}=1$, $c=\frac{\lambda \Delta t}{\Delta x}=c_{M}=1$.} }
	\label{fig:test1_exp_t1}
\end{figure}

\begin{table}
	\begin{tabular}{|c||c|c|c|c|c|c||c|}
		\hline Cells & $\textrm{Err}_{\mathcal{L}^1}$ & $\mathcal{O}_{\mathcal{L}^1}$ & $\textrm{Err}_{\mathcal{L}^2}$ & $\mathcal{O}_{\mathcal{L}^2}$  & $\textrm{Err}_{\mathcal{L}^{\infty}}$ & $\mathcal{O}_{\mathcal{L}^{\infty}}$  &  $r$\\ \hline
		\hline $8$ &$ 	2.15E-2	$ & $		$ & $	2.17E-02	$ & $		$ & $	2.95E-02	$ & $		$&  $ -\frac{1}{8}$\\
		\hline $16$& $	7.10E-3	$ & $	1.6	$ & $	6.97E-3	$ & $	1.64	$ & $	1.03E-2	$ & $	1.52	$&  $-\frac{1}{16} $\\
		\hline $32$&$	1.95E-3	$ & $	1.87	$ & $	1.86E-3	$ & $	1.91	$ & $	2.77E-3	$ & $	1.9	$ &  $-\frac{1}{32}$\\
		\hline $64$& $	5.02E-4	$ & $	1.96	$ & $	4.73E-4	$ & $	1.98	$ & $	7.00E-4	$ & $	1.99	$ &  $ -\frac{1}{64} $\\
		\hline $128$& $	1.27E-4	$ & $	1.99	$ & $	1.18E-4	$ & $	2.0	$ & $	1.74E-4	$ & $	2.01	$ &  $-\frac{1}{128} $\\
		\hline $256$& $	3.19E-5	$ & $	2.0	$ & $	2.96E-5	$ & $	2.0	$ & $	4.33E-5	$ & $	2.01	$ & $-\frac{1}{256}$\\
		\hline $512$& $	7.98E-6	$ & $	2.0	$ & $	7.40E-6	$ & $	2.0	$ & $	1.08E-5	$ & $	2.0	$ & $-\frac{1}{512}$\\
		\hline 
	\end{tabular}
	\caption{Test 1.1. Columns from second to seventh show the errors and convergence rates obtained. The last column depicts the values obtained for $r=\beta \Delta t=-\Delta t$. $\Omega=[0,2]$, $t_{\textrm{end}}=1$, $c=\frac{\lambda \Delta t}{\Delta x}=c_{M}=1$.}\label{tab:test_sol_exacta}
\end{table}

% % % % % % % % % % % % % % % % % % % % % % % % % % % % % %
%                   TEST 1.3                              %
% % % % % % % % % % % % % % % % % % % % % % % % % % % % % %
\subsubsection{Test 1.2.}
The second test analysed present a discontinuity {\color{black}in} the initial conditions:
\begin{eqnarray}
	\partial_t q (x,t) + \frac{1}{2}\partial_x  q(x,t) = -q(x,t), \notag \\
	q(x,0)=\left\lbrace
	\begin{array}{lr}
		1 & x\in\left[\frac{1}{8},\frac{1}{2}\right], \\
		0 & x\in \left[0,\frac{1}{8}\right)\cup \left(\frac{1}{2},\frac{3}{2}\right]. 
	\end{array}\right.\label{eq:grp_step}
\end{eqnarray}
Its exact solution reads
\begin{equation*}
	q(x,t)=\left\lbrace
	\begin{array}{lr}
		1 & x-\frac{1}{2}t\in\left[\frac{1}{8},\frac{1}{2}\right], \\
		0 & x-\frac{1}{2}t\in \left[0,\frac{1}{8}\right)\cup \left(\frac{1}{2},\frac{3}{2}\right].
	\end{array}\right.
\end{equation*}

In Figure \ref{fig:solucion_grp_step} we can observe that the {\color{black}loss} of monotonicity of the scheme produces oscillations near the discontinuity.
This problem arises from considering centred slopes, \eqref{eq:centred_slopes}, which provided a linear scheme. Indeed, we need to circumvent Godunov's theorem to obtain a monotone scheme. This can be done by including non-linear slopes.
In the {\color{black} existing literature}, the linear advection equation case was already studied  combining ADER schemes with ENO, WENO or WAF {\color{black} approaches} obtaining good results.

\begin{figure}
	\centering
	\includegraphics[width=\linewidth]{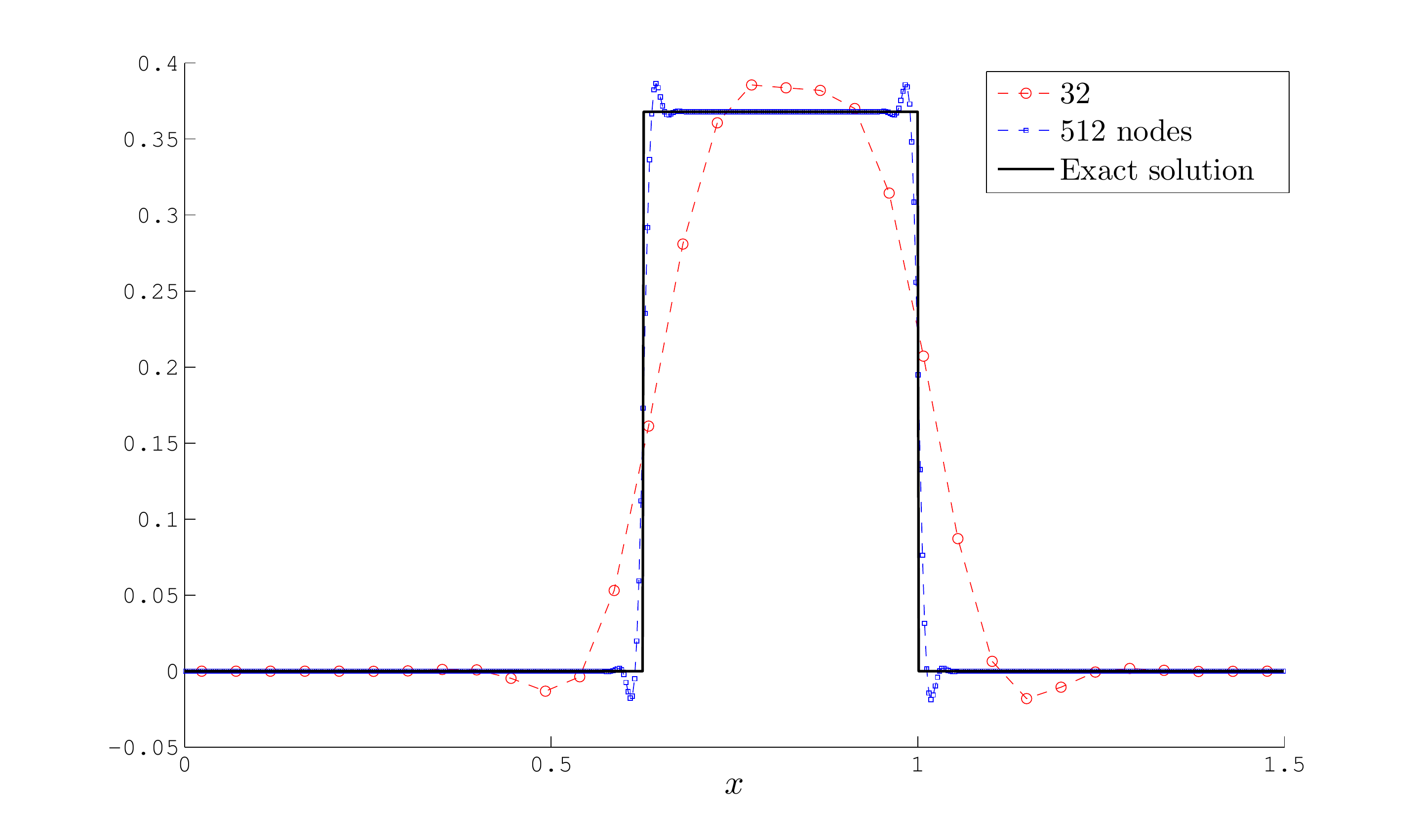}
	\caption{Test 1.2. Exact solution and numerical results obtained for the meshes with 32 and 512 nodes. $\Omega=[0,1.5]$, $t_{\textrm{end}}=1$, $c_{M}=0.5$, $r_{m}=-1$.}
	\label{fig:solucion_grp_step}
\end{figure}

% % % % % % % % % % % % % % % % % % % % % % % % % % % % % %
%                    TEST 3                               %
% % % % % % % % % % % % % % % % % % % % % % % % % % % % % %
\subsection{{\color{black} Test 2. Advection-diffusion-reaction equation}}
{\color{black} Next}, we consider two initial value problem{\color{black}s} for the advection-diffusion-reaction equation.

% % % % % % % % % % % % % % % % % % % % % % % % % % % % % %
%                   TEST 3.1                              %
% % % % % % % % % % % % % % % % % % % % % % % % % % % % % %
\subsubsection{Test 2.1.} 
Following \cite{TM14}, we set a problem with constant diffusion coefficient:
\begin{eqnarray*}
	\partial_t q (x,t) + 10  \partial_x  q(x,t) - 10^{-5} \partial_x^{(2)} q(x,t) = -5  q(x,t), \\
	q(x,0)=\sin(\pi x)
\end{eqnarray*}
in the computational domain  $\Omega\times T=[-1,1]\times [0,1]$ {\color{black} with Dirichlet boundary conditions}.
The exact solution reads
\begin{equation*}
	q(x,t)=\exp( (-\alpha \pi^2+\beta) t )  \sin( \pi (x-\lambda t) ).
\end{equation*}
%The numerical results obtained are detailed in Table \ref{tab:test_completo_tm14}. We can observe that over 256 nodes the approximated solution is very close to the exact solution which results on a stabilization of the error due to the machine precision. Above this mesh refinement, we attain the second order expected.
{\color{black} The numerical results obtained are detailed in Table \ref{tab:test_completo_tm14}. 
%The absolute errors and related orders of accuracy are detailed in Table \ref{tab:test_completo_tm14}.
As the magnitude of the solution is small, relative errors and orders of accuracy are also computed and depicted in Table \ref{tab:test_completorel_tm14} to facilitate the analysis of the results. The expected second order is attained. The lose of accuracy in infinity norm for the two finer meshes is due to the boundary condition approach.}

%\begin{table}
%	\begin{center}
%		\begin{tabular}{|c||c|c|c|c|c|c|}
%			\hline Cells & $\textrm{Err}_{\mathcal{L}^1}$ & $\mathcal{O}_{\mathcal{L}^1}$ & $\textrm{Err}_{\mathcal{L}^2}$ & $\mathcal{O}_{\mathcal{L}^2}$  & $\textrm{Err}_{\mathcal{L}^{\infty}}$ & $\mathcal{O}_{\mathcal{L}^{\infty}}$  \\ \hline
%			\hline $8$ &$ 	1.08E-01	$ & $		$ & $	7.75E-02	$ & $		$ & $	7.01E-02	$ & $		$\\
%			\hline $16$& $	1.83E-02	$ & $	2.5566	$ & $	1.38E-02	$ & $	2.4859	$ & $	1.34E-02	$ & $	2.3929	$\\
%			\hline $32$&$	3.39E-03	$ & $	2.4353	$ & $	2.61E-03	$ & $	2.4058	$ & $	2.58E-03	$ & $	2.3725	$\\
%			\hline $64$& $	6.87E-04	$ & $	2.3008	$ & $	5.35E-04	$ & $	2.2873	$ & $	5.34E-04	$ & $	2.2732	$\\
%			\hline $128$& $	1.40E-04	$ & $	2.2921	$ & $	1.10E-04	$ & $	2.2837	$ & $	1.10E-04	$ & $	2.2814	$\\
%			\hline \rowcolor{black!20} $256$& $	1.94E-05	$ & $	2.857	$ & $	1.52E-05	$ & $	2.854	$ & $	1.52E-05	$ & $	2.854	$\\
%			\hline $512$& $	9.58E-06	$ & $	1.0156	$ & $	7.52E-06	$ & $	1.0136	$ & $	7.52E-06	$ & $	1.0131	$\\
%			\hline
%		\end{tabular}
%		\caption{{\color{black} Test 2.1.} Errors and convergence rates obtained. $\Omega=[-1,1]$, $t_{\textrm{end}}=1$, $c_{M}=0.5$, $d_{M}=0.25$, $r_{m}=-0.25$.}\label{tab:test_completo_tm14}
%	\end{center}
%\end{table}

\begin{table}
	\begin{center}{\color{black}
		\begin{tabular}{|c||c|c|c|c|c|c|}
			\hline Cells & $\textrm{Err}_{\mathcal{L}^1}$ & $\mathcal{O}_{\mathcal{L}^1}$ & $\textrm{Err}_{\mathcal{L}^2}$ & $\mathcal{O}_{\mathcal{L}^2}$  & $\textrm{Err}_{\mathcal{L}^{\infty}}$ & $\mathcal{O}_{\mathcal{L}^{\infty}}$  \\ \hline
			\hline $8$ &  $ 2.76E-04	$ & $		$ & $	4.21E-04	$ & $		$ & $	7.76E-04	$ & $		$\\
			\hline $16$&  $ 1.32E-04 	$ & $	1.2673	$ & $	1.43E-04	$ & $	1.5478	$ & $	2.07E-04	$ & $	1.5336	$\\
			\hline $32$&  $ 4.11E-05  $ & $	1.7220	$ & $   4.01E-05	$ & $	 1.8914 	$ & $	6.59E-05	$ & $	1.9226	$\\
			\hline $64$&  $ 1.25E-05	$ & $	1.8921	$ & $	1.19E-05	$ & $	1.9551 	$ & $	2.48E-05	$ & $	1.9777	$\\
			\hline $128$& $	3.48E-06 	$ & $	1.9437	$ & $	3.24E-06 	$ & $	1.9709	$ & $	6.92E-06	$ & $	 1.9850	$\\
			\hline $256$& $	9.13E-07 	$ & $	1.9347	$ & $	8.39E-07 	$ & $	1.9501	$ & $	1.79E-06 	$ & $	 1.9686	$\\
			\hline $512$& $	2.67E-07	$ & $	1.6799	$ & $	2.38E-07	$ & $	1.7004	$ & $	4.48E-07	$ & $	 1.7376	$\\
			\hline
		\end{tabular}
		\caption{ Test 2.1. Absolute errors and convergence rates obtained. $\Omega=[-1,1]$, $t_{\textrm{end}}=1$, $c_{M}=0.1$, $d_{M}=0.25$, $r_{m}=-0.25$.}\label{tab:test_completo_tm14}   }
	\end{center}
\end{table}

\begin{table}
	\begin{center}{\color{black}
		\begin{tabular}{|c||c|c|c|c|c|c|}
			\hline Cells & $\textrm{Err}_{\textrm{rel}\mathcal{L}^1}$ & $\mathcal{O}_{\textrm{rel}\mathcal{L}^1}$ & $\textrm{Err}_{\textrm{rel}\mathcal{L}^2}$ & $\mathcal{O}_{\textrm{rel}\mathcal{L}^2}$  & $\textrm{Err}_{\textrm{rel}\mathcal{L}^{\infty}}$ & $\mathcal{O}_{\textrm{rel}\mathcal{L}^{\infty}}$  \\ \hline
			\hline $8$ &  $ 5.1E-03$ & $	$ & $	1.0191e-02	$ & $		$ & $2.0383e-02$ & $		$\\
			\hline $16$&  $ 1.34E-03$ & $	2.4668$ & $3.78E-03$ & $1.9668$ & $1.07E-02$ & $1.4668$\\
			\hline $32$&  $ 9.82E-05$ & $	3.2150$ & $3.93E-04$ & $2.7150$ & $1.57E-03$ & $2.2150$\\
			\hline $64$&  $ 6.22E-06$ & $	3.0442$ & $3.52E-05$ & $2.5442$ & $1.99E-04$ & $2.0442$\\
			\hline $128$& $	5.25E-07$ & $	2.9952$ & $4.2E-06$ & $2.4952$ & $3.36E-05$ & $1.9952$\\
			\hline $256$& $	1.02E-07$ & $	2.9426$ & $1.16E-06$ & $2.4426$ & $1.31E-05$ & $1.9426$\\
			\hline $512$& $	4.05E-08$ & $	2.5185$ & $6.48E-07$ & $2.0185$ & $1.03E-05$ & $1.5185$\\
			\hline
		\end{tabular}
		\caption{ Test 2.1. Relative errors and convergence rates obtained. $\Omega=[-1,1]$, $t_{\textrm{end}}=1$, $c_{M}=0.1$, $d_{M}=0.25$, $r_{m}=-0.25$.}\label{tab:test_completorel_tm14}   }
	\end{center}
\end{table}

% % % % % % % % % % % % % % % % % % % % % % % % % % % % % %
%                   TEST 2.2                              %
% % % % % % % % % % % % % % % % % % % % % % % % % % % % % %
\subsubsection{Test 2.2.}
We consider the computational domain $\Omega\times T = [0,2\pi]\times [0,1]$ and the initial value problem with a time and space {\color{black} dependent} diffusion coefficient given by
\begin{eqnarray*}
	\partial_t q (x,t) + 10\partial_x  q(x,t) -10^{-5}\partial_x  \left[\exp(x(t-1)^2) \partial_x q(x,t)\right]  = -5q(x,t), \\
	q(x,0)=\exp(\sin^2(x)),\label{eq:adre_test}
\end{eqnarray*}
{\color{black} with periodic boundary conditions}.
The exact solution for this problem is unknown.
Therefore, in order to obtain the error and the order of accuracy, we compare the obtained solutions with a  reference solution computed for a finer mesh (512 cells).

The obtained results, confirming second order of accuracy, are depicted in Table~\ref{tab:test_completo31} and Figure~\ref{fig:ladre_512}.
\begin{table}
	\begin{center}
		\begin{tabular}{|c||c|c|c|c|c|c|}
			\hline Cells & $\textrm{Err}_{\mathcal{L}^1}$ & $\mathcal{O}_{\mathcal{L}^1}$ & $\textrm{Err}_{\mathcal{L}^2}$ & $\mathcal{O}_{\mathcal{L}^2}$  & $\textrm{Err}_{\mathcal{L}^{\infty}}$ & $\mathcal{O}_{\mathcal{L}^{\infty}}$  \\ \hline
			\hline $8$ &$ 	1.39	$ & $		$ & $	7.18E-1	$ & $		$ & $	6.90E-1	$ & $ $ \\
			\hline $16$& $	3.48E-1	$ & $	2.00	$ & $	2.00E-1	$ & $	1.84	$ & $	2.16E-1	$ & $	1.68	$\\
			\hline $32$&$	6.98E-2	$ & $	2.32	$ & $	4.26E-2	$ & $	2.24	$ & $	4.81E-2	$ & $	2.16	$\\
			\hline $64$& $	1.41E-2	$ & $	2.31	$ & $	8.45E-3	$ & $	2.33	$ & $	9.55E-3	$ & $	2.33	$\\
			\hline $128$& $	2.95E-3	$ & $	2.26	$ & $	1.73E-3	$ & $	2.29	$ & $	1.91E-3	$ & $	2.32	$\\
			\hline $256$& $	5.51E-4	$ & $	2.42	$ & $	3.18E-4	$ & $	2.44	$ & $	3.46E-4	$ & $	2.47	$\\
			\hline 
		\end{tabular}
		\caption{{\color{black} Test 2.2.} Errors and convergence rates obtained. $\Omega=[0,2\pi]$, $t_{\textrm{end}}=1$, $c_{M}=0.5$, $d_{M}=0.25$, $r_{m}=-0.5$.}\label{tab:test_completo31}
	\end{center}
\end{table}
\begin{table}
	\begin{center}
		\begin{tabular}{|c||c|c|c|}
			\hline Cells &  $c$ & $d$ & $r$ \\ \hline
			\hline $8$ & $0.5 $ & $6.37E-7 $ & $-1.96E-1 $\\
			\hline $16$&  $0.5 $ & $1.27E-6 $ & $-9.82E-2 $\\
			\hline $32$& $0.5 $ & $2.55E-6 $ & $-4.92E-2 $\\
			\hline $64$&  $0.5 $ & $5.09E-6 $ & $-2.45E-2 $\\
			\hline $128$& $0.5 $ & $1.02E-5 $ & $-1.23E-2 $\\
			\hline $256$&  $0.5 $ & $2.04E-5 $ & $-6.14E-3 $\\
			\hline 
		\end{tabular}
		\caption{{\color{black} Test 2.2.} Values of the parameters $c$, $d$ and $r$ for the computed time step. $\Omega=[0,2\pi]$, $t_{\textrm{end}}=1$, $c_{M}=0.5$, $d_{M}=0.25$, $r_{m}=-0.5$.}\label{tab:test_completo31_param}
	\end{center}
\end{table}
\begin{figure}
	\centering
	\includegraphics[width=\linewidth]{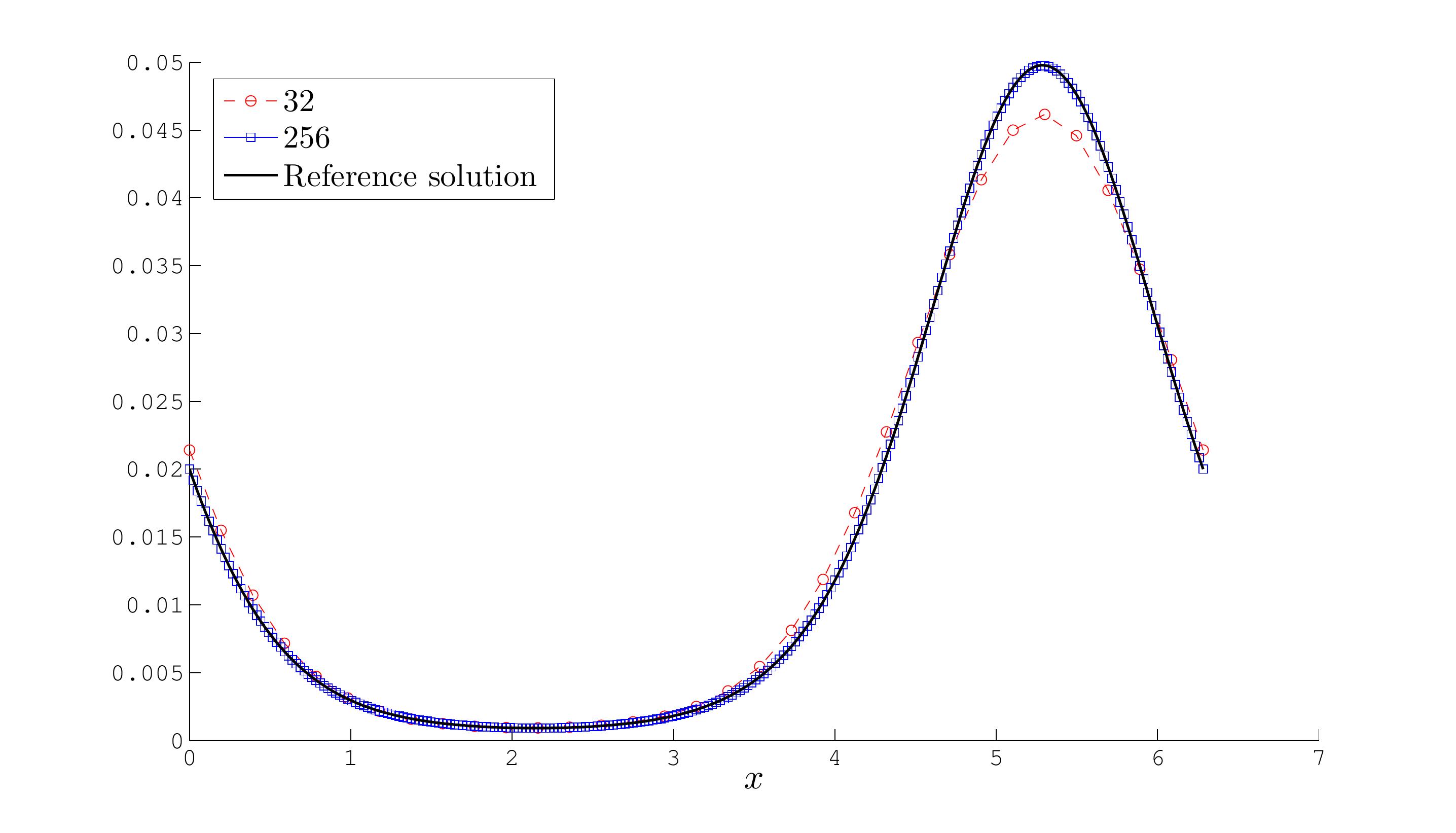}
	\caption{{\color{black}  Test 2.2.} Exact solution and numerical results obtained for the meshes with 32 and 512 nodes. $\Omega=[0,2\pi]$, $t_{\textrm{end}}=1$, $c_{M}=0.5$, $d_{M}=0.25$, $r_{m}=-0.5$.}
	\label{fig:ladre_512}
\end{figure}

% % % % % % % % % % % % % % % % % % % % % % % % % % % % % %
%                   TEST 3                              %
% % % % % % % % % % % % % % % % % % % % % % % % % % % % % %
\subsection{{\color{black} Test 3. Diffusion equation}}
{\color{black} As a final example,} we consider the non-linear diffusion problem proposed in \cite{TH09}:
\begin{eqnarray}
	\partial_t q (x,t) =  \partial_x \left[\left(  q(x,t)\right)^{-1}   \partial_x q(x,t)\right]  , \notag \\
	q(x,0)=\frac{\sinh(2)}{\cosh(2)-\sin\left( \sqrt{2}\left(x-1\right)\right) }
	\label{eq:diffusion_th}
\end{eqnarray}
with periodic boundary conditions in the computational domain $\Omega=[-\sqrt{2}\pi,\sqrt{2}\pi]$. % (period$=\frac{\sqrt{2}}{\pi}$).
Its exact solution reads
\begin{equation*}
	q(x,t)=\frac{\sinh(2t+2)}{\cosh(2t+2)-\sin\left( \sqrt{2}\left(x-1\right)\right)}.
\end{equation*}
The numerical results presented in Table \ref{tab:diffusion_th} confirm second order of accuracy. Figure \ref{fig:diffusion_th} shows the good agreement between the exact solution and the {\color{black} computed solution} for two different meshes (32 and 512 nodes).

\begin{figure}
	\centering
	\includegraphics[width=\linewidth]{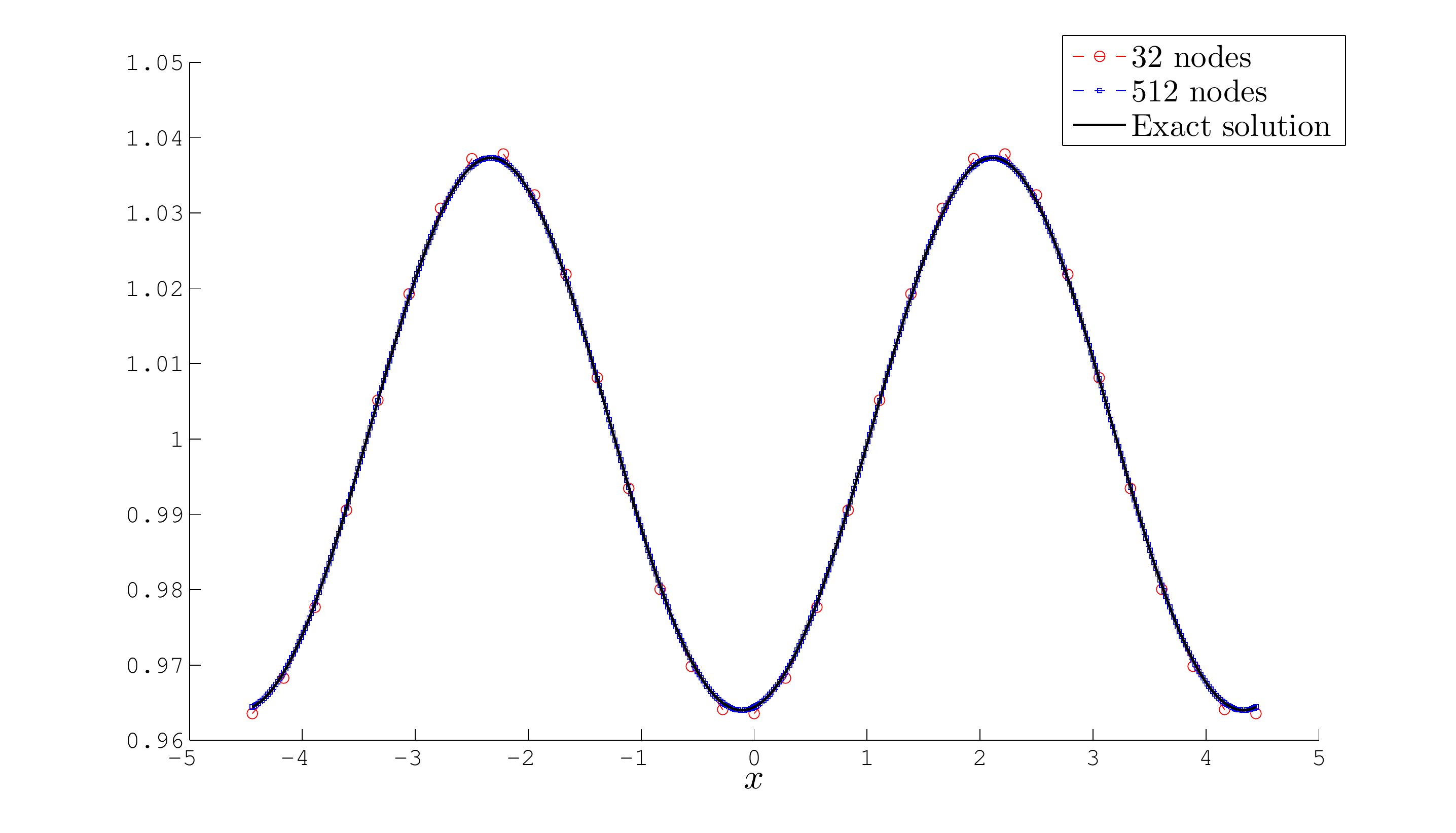}
	\caption{{\color{black} Test 3. } Exact solution and numerical results obtained for the meshes with 32 and 512 nodes. $\Omega=[-\sqrt{2}\pi,\sqrt{2}\pi]$, $t_{\textrm{end}}=1$, $d_{M}=0.25$.}
	\label{fig:diffusion_th}
\end{figure}
\begin{table}
	\begin{tabular}{|c||c|c|c|c|c|c|}
		\hline Cells & $\textrm{Err}_{\mathcal{L}^1}$ & $\mathcal{O}_{\mathcal{L}^1}$ & $\textrm{Err}_{\mathcal{L}^2}$ & $\mathcal{O}_{\mathcal{L}^2}$  & $\textrm{Err}_{\mathcal{L}^{\infty}}$ & $\mathcal{O}_{\mathcal{L}^{\infty}}$  \\ \hline
		\hline $8$ &$ 	1.71E-01	$ & $		$ & $	6.32E-02	$ & $		$ & $	3.57E-02	$ & $		$\\
		\hline $16$& $	3.29E-02	$ & $	2.3751	$ & $	1.19E-02	$ & $	2.4118	$ & $	7.56E-03	$ & $	2.2391	$\\
		\hline $32$&$	7.53E-03	$ & $	2.1268	$ & $	2.79E-03	$ & $	2.0897	$ & $	1.74E-03	$ & $	2.1176	$\\
		\hline $64$& $	1.83E-03	$ & $	2.0413	$ & $	6.83E-04	$ & $	2.0296	$ & $	4.39E-04	$ & $	1.989	$\\
		\hline $128$& $	4.52E-04	$ & $	2.0179	$ & $	1.70E-04	$ & $	2.0103	$ & $	1.09E-04	$ & $	2.0082	$\\
		\hline $256$& $	1.12E-04	$ & $	2.0075	$ & $	4.23E-05	$ & $	2.004	$ & $	2.73E-05	$ & $	1.9993	$\\
		\hline $512$& $	2.80E-05	$ & $ 2.0035	$ & $	1.06E-05	$ & $	 2.0017	$ & $	6.82E-06	$ & $	2.0005	$\\
		\hline 
	\end{tabular}
	\caption{{\color{black} Test 3. } Errors and convergence rates obtained. $\Omega=[-\sqrt{2}\pi,\sqrt{2}\pi]$, $t_{\textrm{end}}=1$, $d_{M}=0.25$.}\label{tab:diffusion_th}
\end{table}

% % % % % % % % % % % % % % % % % % % % % % % % % % % % % %
% % % % % % % % % % % % % % % % % % % % % % % % % % % % % %
%                  Conclusions                            %
% % % % % % % % % % % % % % % % % % % % % % % % % % % % % %
\section{Summary and conclusions}\label{sec:conclusions}
In this paper we have constructed numerical schemes of second order of accuracy in both space and time,
for solving advection-diffusion-reaction partial differential equations. To this end we have adopted the
ADER and the MUSCL-Hancock approaches. Second order of accuracy is ensured by approximating
appropriately the integrals that arise in the finite volume framework. For the model equation we have
performed a detailed linear stability as well as an accuracy analysis in terms of local truncation error.
Empirical convergence rate studies confirm the expected theoretical accuracy analysis. The numerical
schemes studied will prove useful in solving systems of time-dependent advection-diffusion-reaction
equations for realist applications.

% % % % % % % % % % % % % % % % % % % % % % % % % % % % % %
% % % % % % % % % % % % % % % % % % % % % % % % % % % % % %
%                Acknowledgement                          %
% % % % % % % % % % % % % % % % % % % % % % % % % % % % % %
\section*{Acknoledgements}
This work was financially supported by  Spanish MICINN 
projects MTM2008-02483, CGL2011-28499-C03-01 and MTM2013-43745-R; by the Spanish MECD under grant FPU13/00279;
by the Xunta de Galicia Conseller\'ia de Cultura Educaci\'on
e Ordenaci\'on Universitaria under grant \textit{Axudas de apoio
\'a etapa predoutoral do Plan I2C} PRE/2013/031; by Xunta de Galicia and FEDER under research project GRC2013-014 and by Fundaci\'on Barri\'e under grant \textit{Becas de posgrado en el extranjero 2013}.

% % % % % % % % % % % % % % % % % % % % % % % % % % % % % %
% % % % % % % % % % % % % % % % % % % % % % % % % % % % % %
%                  References                             %
% % % % % % % % % % % % % % % % % % % % % % % % % % % % % %
%% Following citation commands can be used in the body text:
%% Usage of \cite is as follows:
%%   \cite{key}         ==>>  [#]
%%   \cite[chap. 2]{key} ==>> [#, chap. 2]
%%

%% References with bibTeX database:

\bibliographystyle{plain}
\bibliography{./mibiblio}

% % % % % % % % % % % % % % % % % % % % % % % % % % % % % %
% % % % % % % % % % % % % % % % % % % % % % % % % % % % % %
%                   Appendix                              %
% % % % % % % % % % % % % % % % % % % % % % % % % % % % % %
%% The Appendices part is started with the command \appendix;
%% appendix sections are then done as normal sections
\appendix
\section{Truncation error}\label{sec:appendix}
In this appendix, the order of accuracy of the attained schemes is analysed. Two distinct cases are studied: the linear ad\-vec\-tion-reaction equation and the ad\-vec\-tion-dif\-fu\-sion-reac\-tion equation with time and space dependent diffusion coefficient.

\subsection{Truncation error of the linear advection-reaction equation}\label{apendixa1}
Using Taylor series expansion, second order in space and time can be proved for scheme \eqref{eq:lare_num_scheme} with centred slopes: 
\begin{align*}
	\tau_{j}^{n} &= & &\frac{1}{\Delta t} \left[ q(x_j,t^{n+1}) - q(x_{j},t^{n}) \right] 
	+\frac{\lambda}{\Delta x}
	\left\lbrace \phantom{\frac{t}{b}}\!\!\!  \left[ q(x_{j},t^{n}) - q(x_{j-1},t^{n}) \right]\right. \\
	&& +&\left.  \frac{1}{4} \left[  q(x_{j+1},t^{n})- q(x_{j},t^{n})- q(x_{j-1},t^{n})+ q(x_{j-2},t^{n})\right] \right. \\
	& & +&\left. \frac{\Delta t}{2} \left[-\frac{\lambda}{2\Delta x} \left[  q(x_{j+1},t^{n})- q(x_{j},t^{n})- q(x_{j-1},t^{n})+ q(x_{j-2},t^{n}) \right] \right.\right.\\
	& &+&\left. \left. \frac{\beta}{\Delta x} \left(\phantom{\frac{t}{b}}\!\!\! \left[ q(x_{j},t^{n})-q(x_{j-1},t^{n})\right]  \right. \right.\right.\\
	&   &+&\left. \left.\left.\frac{ 1}{4} \left[ q(x_{j+1},t^{n})- q(x_{j},t^{n})- q(x_{j-1},t^{n})+ q(x_{j-2},t^{n})\right] \right) \right] 
	\right\rbrace \\
	& & - &\frac{\beta}{\Delta x} \left[q(x_{j},t^{n})+\frac{\Delta t}{2}\left(-\frac{\lambda}{2\Delta x} \left(q(x_{j+1},t^{n})-q(x_{j-1},t^{n}) \right) +\beta q(x_{j},t^{n})\right) \right]\\
	&=  &  &\partial_t q(x_j,t^n) +\lambda \partial_x q(x_j,t^n) - \beta q(x_j,t^n)   \\
	& & + &\lambda\left[ -\frac{1}{2} \partial^{(2)}_x q(x_j,t^n) \Delta x + \frac{1}{6} \partial^{(3)}_x q(x_j,t^n) \Delta x^2 \right. \\
	& & +&\left.\frac{1}{4}\left[  2\partial^{(2)}_x q(x_j,t^n)\Delta x - \partial^{(3)}_x q(x_j,t^n)\Delta x^2   \right] \right]
	\\
	& & +& \Delta t\left\lbrace \frac{1}{2} \partial^{(2)}_t q(x_j,t^n)
	+  \frac{\lambda}{2} \left[-\frac{\lambda}{2} \left[   2\partial^{(2)}_x q(x_j,t^n) - \partial^{(3)}_x q(x_j,t^n)\Delta x \right]  \right. \right.\\
	& & +&\beta \left( \partial_x q(x_j,t^n) -\frac{1}{2} \partial^{(2)}_x q(x_j,t^n) \Delta x + \frac{1}{6} \partial^{(3)}_x q(x_j,t^n) \Delta x^2  \right.\\
	& &  +&\left. \left. \frac{1}{4} \left[ 2\partial^{(2)}_x q(x_j,t^n)\Delta x - \partial^{(3)}_x q(x_j,t^n)\Delta x^2 \right]   \right) \right]  
	\\
	& & -&\left. \frac{\beta}{2}\left[-\frac{\lambda}{2\Delta x}\left( 2\partial_x q(x_j,t^n)\Delta x + \frac{1}{3}\partial^{(3)}_x q(x_j,t^n)\Delta x^3 \right) +\beta q(x_j,t^n)\right] \right\rbrace 
	\\
	& & +&\mathcal{O}\left(\Delta t^2 \right) +\mathcal{O}\left(\Delta x^2 \right) +\mathcal{O}\left(\Delta x \Delta t \right) 
	\\
	& =  & & \frac{\Delta t}{2} \left[ \partial^{(2)}_t q(x_j,t^n)
	-\lambda^2  \partial^{(2)}_x q(x_j,t^n) +2\lambda \beta  \partial_x q(x_j,t^n) -\beta^2 q(x_j,t^n)\right]\\
	& & +&\mathcal{O}\left(\Delta t^2 \right) +\mathcal{O}\left(\Delta x^2 \right)+\mathcal{O}(\Delta x \Delta t )
	\\
	& = & & \mathcal{O}\left(\Delta t^2 \right) +\mathcal{O}\left(\Delta x^2 \right)+\mathcal{O}(\Delta x \Delta t ).
\end{align*}
The last equality arises from
\[ \partial^{(2)}_t q(x_j,t^n)
-\lambda^2  \partial^{(2)}_x q(x_j,t^n) +\lambda \beta  \partial_x q(x_j,t^n)=  -\lambda \beta  \partial_x q(x_j,t^n) -\beta^2 q(x_j,t^n)  \]
which is obtained following the Cauchy-Kovalevskaya procedure.

\subsection{Truncation error of the advection-diffusion-reaction equation}
The truncation error of the scheme for the  advection-diffusion-reaction equation with time and space dependent diffusion coefficient is given by
\begin{align*}
	\tau_{j}^{n} &= & &\frac{1}{\Delta t} \left[ q(x_j,t^{n+1}) - q(x_{j},t^{n}) \right] 
	+\frac{\lambda}{\Delta x}
	\left\lbrace \phantom{\frac{t}{b}}\!\!\!\left[  q(x_{j},t^{n}) - q(x_{j-1},t^{n}) \right] \right. \\
	&& +&\left.  \frac{1}{4} \left[   q(x_{j+1},t^{n})- q(x_{j},t^{n})- q(x_{j-1},t^{n})+ q(x_{j-2},t^{n})\right]  \right. \\
	& & +&\left. \frac{\Delta t}{2} \left[-\frac{\lambda}{2\Delta x} \left[  q(x_{j+1},t^{n})- q(x_{j},t^{n})- q(x_{j-1},t^{n})+ q(x_{j-2},t^{n}) \right]  \right.\right.\\
	& & +&  \frac{1}{\Delta x} \left[ \left(
	\alpha(x_{j+\frac{1}{2}},t^{n})\frac{q(x_{j+1},t^{n})-q(x_{j},t^{n}) }{\Delta x} - \alpha(x_{j-\frac{1}{2}},t^{n})\frac{q(x_{j},t^{n})-q(x_{j-1},t^{n}) }{\Delta x} \right) \right.  \\[6pt] & & - &\left. \left( 
	\alpha(x_{j-\frac{1}{2}},t^{n})\frac{q(x_{j},t^{n})-q(x_{j-1},t^{n}) }{\Delta x} - \alpha(x_{j-\frac{3}{2}},t^{n})\frac{q(x_{j-1},t^{n})-q(x_{j-2},t^{n}) }{\Delta x} \right)
	\right]\\[6pt]
	& &+&\left. \left. \frac{\beta}{\Delta x} \left(\phantom{\frac{t}{b}}\!\!\!  \left[ q(x_{j},t^{n})-q(x_{j-1},t^{n})\right]+\frac{ 1}{4} \left[  q(x_{j+1},t^{n})- q(x_{j},t^{n})  \right. \right. \right.\right.\\
	&   &-&\left. \left.\left.\left. q(x_{j-1},t^{n})+ q(x_{j-2},t^{n})\right] \phantom{\frac{b}{b}}\!\!\! \right) \right] 
	\right\rbrace 
	-\frac{1}{\Delta x^2} \left\lbrace 
	\overline{\alpha}(x_{j+\frac{1}{2}},t^n) \left[
	\phantom{\frac{b}{b}}\!\!\!
	q(x_{j+1},t^n)-q(x_{j},t^{n})\right.\right.
	\\& & &\left. \left. +\frac{\Delta t}{2}\left(
	\phantom{\frac{b}{b}}\!\!\!\!
	-\lambda \left[ q(x_{j+2},t^{n})
	- 2q(x_{j+1},t^{n})+q(x_{j},t^{n})\right] 
	+\frac{1}{\Delta x^2}\alpha(x_{j+\frac{3}{2}},t^n)
	\right. \right.\right.\\ & & & 
	\left. \left. \left.  \left[  q(x_{j+2},t^{n})-
	q(x_{j+1},t^{n})\right] 
	-\frac{1}{\Delta x^2}\alpha(x_{j+\frac{1}{2}},t^n)\left[  2q(x_{j+1},t^{n})-2q(x_{j},t^{n})\right] 
	\right.\right.\right.\\ & & +& \left.\left.\left.
	\frac{1}{\Delta x^2}\alpha(x_{j-\frac{1}{2}},t^n)\left[  q(x_{j},t^{n})-q(x_{j-1},t^{n})\right] 
	+ \beta\left[ q(x_{j+1},t^{n})-q(x_{j},t^{n})\right]  \phantom{\frac{b}{b}}\!\!\!\right) 
	\right] 
	\right.\\ & & +& \left.
	\overline{\alpha}(x_{j-\frac{1}{2}},t^n) \left[
	\phantom{\frac{b}{b}}\!\!\!
	q(x_{j-1},t^{n})-q(x_{j},t^{n})
	+ \frac{\Delta t}{2}\left(\phantom{\frac{b}{b}}\!\!\!
	-\lambda \left[ -q(x_{j+1},t^{n})+2q(x_{j},t^{n})  \right.\right. \right.\right.\\ & & -&
	\left. \left. \left. \left.q(x_{j-1},t^{n})\right]-
	\frac{1}{\Delta x^2}\alpha(x_{j+\frac{1}{2}},t^n)\left[  q(x_{j+1},t^{n})-q(x_{j},t^{n})\right] 
	\right. \right.\right.\\  & & +&\left. \left. \left. 
	\frac{1}{\Delta x^2}\alpha(x_{j-\frac{1}{2}},t^n)\left[  2q(x_{j},t^{n})-2q(x_{j-1},t^{n})\right] 
	-\frac{1}{\Delta x^2}\alpha(x_{j-\frac{3}{2}},t^n)\left[  q(x_{j-1},t^{n})\right.
	\right.\right.\right. \\ & & -&\left.\left.\left.
	\left.q(x_{j-2},t^{n})\right] 
	+\beta\left[ q(x_{j-1},t^{n})-q(x_{j},t^{n})\right] \phantom{\frac{b}{b}}\!\!\!\right) 
	\right] 
	\right\rbrace\\[6pt]
	& & - &\frac{\beta}{\Delta x} \left\lbrace q(x_{j},t^{n})+\frac{\Delta t}{2}\left[ -\frac{\lambda}{2\Delta x} \left[ q(x_{j+1},t^{n})-q(x_{j-1},t^{n})\right]  \right.\right. \\& &+&\left.\left.\frac{1}{\Delta x^2} \left(\alpha(x_{j+\frac{1}{2}},t^n)\left[ q(x_{j+1},t^n) -q(x_{j},t^n)\right]  - \alpha(x_{j-\frac{1}{2}},t^n) \left[ q(x_{j},t^n)-q(x_{ j-1},t^n) \right] \right) \right.\right. \\& &+&\left.\left. \frac{\beta}{2\Delta x} q(x_{j},t^{n})\right] \right\rbrace
\end{align*}
where
\begin{equation*}
	\overline{\alpha}(x_{j+\frac{1}{2}},t^n)=\alpha(x_{j+\frac{1}{2}},t^n) + \frac{\Delta t}{2}\partial_t \alpha(x_{j+\frac{1}{2}},t^n),
\end{equation*}
\begin{equation*}
	\overline{\alpha}(x_{j-\frac{1}{2}},t^n)=\alpha(x_{j-\frac{1}{2}},t^n) + \frac{\Delta t}{2}\partial_t \alpha(x_{j-\frac{1}{2}},t^n).
\end{equation*}
Then, we can proceed analysing {\color{black} each of the} terms which depend on the diffusion term: 
\begin{itemize}
	\item Local truncation error contribution of the diffusion term to the flux term:
\end{itemize}
\begin{align*}
	& & & \frac{\lambda \Delta t}{2 \Delta x^2} \left\lbrace\left[ 
	\alpha(x_{j+\frac{1}{2}},t^{n})\frac{q(x_{j+1},t^{n})-q(x_{j},t^{n}) }{\Delta x} - \alpha(x_{j-\frac{1}{2}},t^{n})\frac{q(x_{j},t^{n})-q(x_{j-1},t^{n}) }{\Delta x} \right] \right. \\& & -&\left.\left[ 
	\alpha(x_{j-\frac{1}{2}},t^{n})\frac{q(x_{j},t^{n})-q(x_{j-1},t^{n}) }{\Delta x} - \alpha(x_{j-\frac{3}{2}},t^{n})\frac{q(x_{j-1},t^{n})-q(x_{j-2},t^{n}) }{\Delta x} \right]
	\right\rbrace\\
	& = & & \frac{\lambda \Delta t}{2 \Delta x^3} \left\lbrace
	\left[ 
	\alpha(x_{j+\frac{1}{2}},t^{n})\left(   q(x_{j},t^{n})+\partial_x q(x_{j},t^{n}) \Delta x +\frac{1}{2}\partial^{(2)} q(x_{j},t^{n}) \Delta x^2\right.\right.\right.\\ & & + &\left. \left.\left.  
	\frac{1}{6}\partial^{(3)} q(x_{j},t^{n}) \Delta x^3+\mathcal{O}(\Delta x^4)-q(x_{j},t^{n}) \right)
	- \alpha(x_{j-\frac{1}{2}},t^{n})\left( \phantom{\frac{b}{b}}\!\!\! q(x_{j},t^{n})-  q(x_{j},t^{n}) \right.\right.\right.\\ & & +&\left.\left. \left. \partial_x q(x_{j},t^{n}) \Delta x -\frac{1}{2}\partial^{(2)} q(x_{j},t^{n}) \Delta x^2 + \frac{1}{6}\partial^{(3)} q(x_{j},t^{n}) \Delta x^3+\mathcal{O}(\Delta x^4) )\right)  \right] \right. \\ & & -&\left. \left[ 
	\alpha(x_{j-\frac{1}{2}},t^{n})\left( q(x_{j},t^{n})-  q(x_{j},t^{n})+\partial_x q(x_{j},t^{n}) \Delta x -\frac{1}{2}\partial^{(2)} q(x_{j},t^{n}) \Delta x^2 
	\right.\right.\right.\\ & & + &\left.\left.\left.
	\frac{1}{6}\partial^{(3)} q(x_{j},t^{n}) \Delta x^3+\mathcal{O}(\Delta x^4)\right)  - \alpha(x_{j-\frac{3}{2}},t^{n})\left( \phantom{\frac{b}{b}}\!\!\!  q(x_{j},t^{n})-\partial_x q(x_{j},t^{n}) \Delta x
	\right.\right.\right.\\ & & + &\left.\left.\left.
	\frac{1}{2}\partial^{(2)} q(x_{j},t^{n}) \Delta x^2 - 
	\frac{1}{6}\partial^{(3)} q(x_{j},t^{n}) \Delta x^3+\mathcal{O}(\Delta x^4)-q(x_{j},t^{n}) \right.\right.\right.\\& & +&\left.\left.\left.
	2\partial_x q(x_{j},t^{n}) \Delta x -2\partial^{(2)} q(x_{j},t^{n}) \Delta x^2 + 
	\frac{4}{3}\partial^{(3)} q(x_{j},t^{n}) \Delta x^3+\mathcal{O}(\Delta x^4)\right)  \right]
	\right\rbrace \\
	& = & & \frac{\lambda \Delta t}{2\Delta x^2}\left\lbrace
	\partial_x q(x_j,t^n)\left[\alpha(x_{j+\frac{1}{2}},t^{n})-2\alpha(x_{j-\frac{1}{2}},t^{n})+\alpha(x_{j-\frac{3}{2}},t^{n})\right]\right.\\& & & \left.
	+\frac{1}{2} \partial_x^{(2)} q(x_j,t^n)\Delta x
	\left[\alpha(x_{j+\frac{1}{2}},t^{n})+2\alpha(x_{j-\frac{1}{2}},t^{n})-3\alpha(x_{j-\frac{3}{2}},t^{n})\right]\right.\\& & &\left.
	+\frac{1}{6}\partial^{(3)}_x q(x_j,t^n)\Delta x^2
	\left[\alpha(x_{j+\frac{1}{2}},t^{n})-2\alpha(x_{j-\frac{1}{2}},t^{n})+7\alpha(x_{j-\frac{3}{2}},t^{n})\right]
	+\mathcal{O}(\Delta x^4)
	\right\rbrace\\
	& = & & \frac{\lambda \Delta t}{2}\left[   
	\partial_x q(x_j,t^n) \partial_x^{(2)}\alpha(x_j,t^n) +2\partial_x^{(2)} q(x_j,t^n) \partial_x\alpha(x_j,t^n) \right.\\
	& & +& \left. \partial^{(3)}_x q(x_j,t^n) \alpha(x_j,t^n) +\mathcal{O}(\Delta x)  
	\right] \\
	&= & & \frac{\lambda \Delta t}{2} \partial_x\left[\partial_x\left(\alpha(x_j,t^n)\partial_x q(x_j,t^n)\right)\right]+\mathcal{O}(\Delta x\Delta t).
\end{align*}

\begin{itemize}  
	\item Local truncation error contribution of the diffusion term:
\end{itemize}
\begin{align*}
	& & - &\frac{1}{\Delta x^2} \left\lbrace 
	\overline{ \alpha}(x_{j+\frac{1}{2}},t^n)   \left[ \phantom{\frac{b}{b}}\!\!\!
	q(x_{j+1},t^n)-q(x_{j},t^{n}) \right.\right.\\ & & +& \left. \left. \frac{\Delta t}{2}\left( \phantom{\frac{b}{b}}\!\!\! -\lambda \left[ q(x_{j+2},t^{n})-2q(x_{j+1},t^{n})+q(x_{j},t^{n})\right]  \right. \right.\right.\\
	& &
	+&   \left. \left. \left. \frac{1}{\Delta x^2}\alpha(x_{j+\frac{3}{2}},t^n)\left[  q(x_{j+2},t^{n})-q(x_{j+1},t^{n})\right]
	\right.\right.\right. \\
	& &
	- & \left.\left.\left. \frac{1}{\Delta x^2}\alpha(x_{j+\frac{1}{2}},t^n)\left[  2q(x_{j+1},t^{n})-2q(x_{j},t^{n})\right]
	\right.\right.\right. \\ 
	& &+& \left.\left.\left. \frac{1}{\Delta x^2}\alpha(x_{j-\frac{1}{2}},t^n)\left[  q(x_{j},t^{n})-q(x_{j-1},t^{n})\right]
	\right.\right.\right. \\   
	& &+ & \left.\left.\left.  \beta\left[ q(x_{j+1},t^{n})-q(x_{j},t^{n})\right] \phantom{\frac{b}{b}}\!\!\!\right) 
	\right] 
	\right.\\ & & + &\left.
	\overline{\alpha}(x_{j-\frac{1}{2}},t^n)  \left[\phantom{\frac{b}{b}}\!\!\!\
	q(x_{j-1},t^{n})-q(x_{j},t^{n})
	\right. \right.\\ & & +& \left. \left.
	\frac{\Delta t}{2}\left( \phantom{\frac{b}{b}}\!\!\!\! -\lambda \left[ -q(x_{j+1},t^{n})+2q(x_{j},t^{n})-q(x_{j-1},t^{n})\right]  \right. \right.\right.\\ & & -  &
	\left. \left. \left.  
	\frac{1}{\Delta x^2}\alpha(x_{j+\frac{1}{2}},t^n)\left[  q(x_{j+1},t^{n})-q(x_{j},t^{n})\right]
	\right. \right.\right.\\ & & + & \left. \left. \left. 
	\frac{1}{\Delta x^2}\alpha(x_{j-\frac{1}{2}},t^n)\left[  2q(x_{j},t^{n})-2q(x_{j-1},t^{n})\right]
	\right.\right.\right. \\ & & -&\left.\left.\left.
	\frac{1}{\Delta x^2}\alpha(x_{j-\frac{3}{2}},t^n)\left[  q(x_{j-1},t^{n})-q(x_{j-2},t^{n})\right]
	\right.\right.\right. \\ & & +&\left.\left.\left.
	\beta\left[ q(x_{j-1},t^{n})-q(x_{j},t^{n})\right] \phantom{\frac{b}{b}}\!\!\!\right) 
	\right] 
	\right\rbrace\\
	&=&-&\left[  \partial_x \overline{\alpha} (x_j,t^n)\right) \left(\partial_x q(x_j,t^n) \right]  - \frac{\overline{\alpha} (x_{j+1},t^n)+ \overline{\alpha} (x_{j-1},t^n)}{2}\partial^{(2)}_x q(x_{j},t^n)\\
	& &  -& \lambda\frac{\Delta t}{2} \left( \partial_x \overline{\alpha} (x_j,t^n)\right)\left(\partial^{(2)}_x q(x_{j},t^n) \right) + \lambda\frac{\Delta t}{2} \overline{\alpha} (x_j,t^n)\left( \partial^{(3)}_x q(x_{j},t^n) \right)\\
	& & - &\frac{\Delta t}{2} \left[ \phantom{\frac{b}{b}}\!\!\! \left(\partial_x q(x_{j},t^n) \right) \left(\partial_x^{(2)} \alpha (x_j,t^n) \right) \left(\partial_x \overline{\alpha} (x_j,t^n) \right) \right.\\
	& & + & \left. \overline{\alpha} (x_j,t^n) \left(\partial_x q(x_{j},t^n) \right) \left( \partial_x^{(3)} \alpha (x_j,t^n)\right)
	\right.\\& &  +&\left.
	2 \left(\partial_x \overline{\alpha} (x_j,t^n) \right) \left(\partial^{(2)}_x q(x_{j},t^n)\right) \left(\partial_x \alpha (x_j,t^n) \right) 
	\right.\\& &  +&\left.
	3  \overline{\alpha} (x_j,t^n) \left(\partial^{(2)}_x q(x_{j},t^n)\right) \left( \partial_x^{(2)} \alpha (x_j,t^n)\right)
	\right.\\ & & +& \left.
	\left( \partial_x \overline{\alpha} (x_j,t^n)\right) \left(\partial^{(3)}_x q(x_{j},t^n)\right)\alpha (x_j,t^n)
	\right.\\& &  +&\left.
	3 \overline{\alpha} (x_j,t^n) \left(\partial^{(3)}_x q(x_{j},t^n)\right)\left( \partial_x \alpha (x_j,t^n)\right)
	\right.\\ & & +& \left.
	\overline{\alpha} (x_j,t^n) \left(\partial^{(4)}_x q(x_{j},t^n)\right) \alpha (x_j,t^n)  \phantom{\frac{b}{b}}\!\!\! \right] 
	-\beta \frac{\Delta t}{2} \left[\phantom{\frac{b}{b}}\!\!\! \left( \partial_x \overline{\alpha} (x_j,t^n)\right) \left(\partial_x q(x_{j},t^n)\right) 
	\right.\\ & & +&\left.
	\overline{\alpha} (x_j,t^n) \left(\partial^{(2)}_x q(x_{j},t^n)\right) \phantom{\frac{b}{b}}\!\!\!\right]+\mathcal{O}(\Delta x^2 ) +\mathcal{O}(\Delta x \Delta t) 
	\\
	&=& - &\partial_x\left(\alpha(x_j,t^n)\partial_x q(x_j,t^n)\right) + \frac{\Delta t}{2}\left\lbrace \phantom{\frac{b}{b}}\!\!\! \partial_x\left(\partial_t\alpha(x_j,t^n)\partial_x q(x_j,t^n)\right) \right. \\& & -&\left.
	\lambda \partial_x\left(\alpha(x_j,t^n)\partial^{(2)}_x q(x_j,t^n)\right) + \partial_x\left[\alpha(x_j,t^n)\partial^{(2)}_x\left(  \alpha(x_j,t^n) \partial_xq(x_j,t^n)\right) \right]
	\right. \\ & & +&\left.
	\beta \partial_x\left(\alpha(x_j,t^n)\partial_x q(x_j,t^n)\right) \phantom{\frac{b}{b}}\!\!\!
	\right\rbrace 
	+\mathcal{O}(\Delta x^2 )+\mathcal{O}(\Delta x \Delta t ).
\end{align*}

\begin{itemize}    
	\item Local truncation error contribution of the diffusion term to the source term:
\end{itemize}
\begin{align*}
	& & -&  \frac{\beta\Delta t}{2\Delta x^2} \left\lbrace\alpha(x_{j+\frac{1}{2}},t^n)\left[ q(x_{j+1},t^n) -q(x_{j},t^n)\right]\right.  \\ & &- &\left. \alpha(x_{j-\frac{1}{2}},t^n) \left[ q(x_{j},t^n)-q(x_{ j-1},t^n) \right] \right\rbrace \\ 
	&= &- &\frac{\beta\Delta t}{2\Delta x^2} \left[\alpha(x_{j+\frac{1}{2}},t^n)\left(  q(x_{j},t^{n})+\partial_x q(x_{j},t^{n}) \Delta x +\frac{1}{2}\partial^{(2)} q(x_{j},t^{n}) \Delta x^2 
	\right.\right.\\ & &+  & \left.\left.\frac{1}{6}\partial^{(3)} q(x_{j},t^{n}) \Delta x^3+\mathcal{O}(\Delta x^4) -q(x_{j},t^n)\right)
	\right.\\ & & -& \left.
	\alpha(x_{j-\frac{1}{2}},t^n) \left(\phantom{\frac{b}{b}}\!\!\! q(x_{j},t^n)-  q(x_{j},t^{n})-\partial_x q(x_{j},t^{n}) \Delta x  \right.\right.\\ & & +&\left.\left.  
	\frac{1}{2}\partial^{(2)} q(x_{j},t^{n}) \Delta x^2-\frac{1}{6}\partial^{(3)} q(x_{j},t^{n}) \Delta x^3+\mathcal{O}(\Delta x^4) \right) \right] \\
	&=&-&\frac{\beta\Delta t}{2\Delta x}\left[
	\partial_x q(x_j,t^n)\left(\alpha(x_{j+\frac{1}{2}},t^n)-\alpha(x_{j-\frac{1}{2}},t^n)\right) \right.\\ & & +&\left.
	\frac{1}{2}\partial_x^{(2)} q(x_j,t^n)\Delta x \left(\alpha(x_{j+\frac{1}{2}},t^n)+\alpha(x_{j-\frac{1}{2}},t^n)\right) \right.\\ & & + &\left.
	\frac{1}{6}\partial_x^{(3)} q(x_j,t^n)\Delta x^2
	\left(\alpha(x_{j+\frac{1}{2}},t^n)-\alpha(x_{j-\frac{1}{2}},t^n)\right) +\mathcal{O}(\Delta x^3)
	\right]\\
	& = & -&\frac{\beta\Delta t}{2}\partial_x\left(\alpha(x_j,t^n)\partial_x q(x_j,t^n)\right) + \mathcal{O}(\Delta x^2)+\mathcal{O}(\Delta x \Delta t ).
\end{align*}
Gathering together the previous terms and the already obtained for the linear advection-reaction equation {\color{black} (\ref{apendixa1})}, we get

\begin{align*}
	\tau^n& =  & & \partial_t q(x_j,t^n) +\lambda \partial_x q(x_j,t^n) -\partial_x\left[ \alpha(x_j,t^n)\partial_x q(x_j,t^n)\right] - \beta q(x_j,t^n) 
	\\
	& & +&\frac{\Delta t}{2} \left[\phantom{\frac{b}{b}}\!\!\! \partial^{(2)}_t q(x_j,t^n) +\lambda
	\partial_x \left[  -\lambda  \partial_x q(x_j,t^n) + \beta q(x_j,t^n)  \right] \right.\\
	& & -&\left. \beta\left[  -\lambda\partial_x q(x_j,t^n)  +\beta q(x_j,t^n) \right] \phantom{\frac{b}{b}}\!\!\! \right]
	-\frac{\Delta t}{2}  \left\lbrace \phantom{\frac{b}{b}}\!\!\!
	\partial_x\left[\partial_t\alpha(x_j,t^n)\partial_x q(x_j,t^n)\right] \right. \\& &-&\left.
	\lambda \partial_x\left[\alpha(x_j,t^n)\partial^{(2)}_x q(x_j,t^n)\right] + \partial_x\left\lbrace\alpha(x_j,t^n)\partial^{(2)}_x\left[  \alpha(x_j,t^n) \partial_xq(x_j,t^n)\right] \right\rbrace
	\right. \\ & &+&\left.
	\beta \partial_x\left[\alpha(x_j,t^n)\partial_x q(x_j,t^n)\right]
	\phantom{\frac{b}{b}}\!\!\! \right\rbrace 
	-\frac{\beta\Delta t}{2}\partial_x\left[ \alpha(x_j,t^n)\partial_x q(x_j,t^n)\right]     \\
	& & + &\frac{\Delta t}{2}
	\partial^{(2)}_x\left[\alpha(x_j,t^n)\partial_x q(x_j,t^n)\right] + \mathcal{O}\left(\Delta t^2 \right) +\mathcal{O}\left(\Delta x^2 \right) +\mathcal{O}\left(\Delta x \Delta t \right)
	\\	& =  & &\mathcal{O}\left(\Delta t^2 \right) +\mathcal{O}\left(\Delta x^2 \right)+\mathcal{O}(\Delta x \Delta t ).
\end{align*} 
Where we have take into account that, following Cauchy-Kovalevskaya,
\begin{gather*}
	\partial_t^{(2)} q
	- \lambda^2 \partial^{(2)}_x q
	-\partial_{x}\left[ \alpha \partial_{x}^{(2)}\left(  \alpha \partial_{x} q\right)  \right] 
	- \beta^2 q
	+2\lambda \beta \partial_x  q  \notag
	- 2   \beta \partial_{x}\left( \alpha\partial_{x} q \right)
	\\-\partial_{x}\left[\left( \partial_t \alpha\right) \left(\partial_x q \right)  \right] 
	+\lambda\partial_{x}^{(2)}\left(  \alpha \partial_{x} q \right)
	+\lambda\partial_x\left(  \alpha \partial_{x}^{(2)} q\right)  
	=0.
\end{gather*}
Thus, we conclude that the scheme is second order in space and time.

%%%%%%%%%%%%%%%%%%%%%%%%%%%%%%%%%%%%%%%%%%%%%%%%%%%%%%%%
%%%%%%%%%%%%%%%%%%%%%%%%%%%%%%%%%%%%%%%%%%%%%%%%%%%%%%%%
\end{document}